\documentclass{amsart}

\usepackage[letterpaper,margin=1in]{geometry}

\usepackage{amsmath,amssymb,amsthm}
\usepackage{hyperref}
\usepackage{cleveref}
\usepackage{mathtools}
\usepackage{tikz-cd}
\usepackage{enumerate}
\usepackage{xcolor}
\usepackage{dsfont}
\usepackage{bbm}

\newtheorem{theorem}{Theorem}[section]
\newtheorem{proposition}[theorem]{Proposition}
\newtheorem{lemma}[theorem]{Lemma}

\newtheorem{fact}[theorem]{Fact}

\theoremstyle{definition}
\newtheorem{definition}[theorem]{Definition}

\newtheorem{question}[theorem]{Question}

\theoremstyle{remark}
\newtheorem{remark}[theorem]{Remark}

\numberwithin{equation}{section}

\newcommand\C{\mathbb{C}}
\newcommand\N{\mathbb{N}}

\newcommand{\bM}{\mathbb{M}}

\newcommand\cT{\mathcal{T}}

\newcommand{\Cst}{\mathrm{C}^*}

\DeclareMathOperator{\tr}{tr}

\DeclareMathOperator{\id}{id}

\DeclareMathOperator{\Lip}{Lip}

\DeclareMathOperator{\Ran}{Ran}
\DeclareMathOperator{\op}{op}
\DeclareMathOperator{\supp}{supp}

\DeclarePairedDelimiter{\ip}{\langle}{\rangle}
\DeclarePairedDelimiter{\norm}{\lVert}{\rVert}

\newcommand{\freedistance}{W_{H,\wedge}}
\newcommand{\tensordistance}{W_{H,\otimes}}
\newcommand{\ldistance}{W_{H,1}}

\title[Quantum Wasserstein distances for quantum permutation groups]{Quantum Wasserstein distances \\ for quantum permutation groups}

\author{Anshu}
\address{\parbox{\linewidth}{Fields Institute for Research in Mathematical Sciences\\
		222 College Street, Toronto, Ontario, Canada}}
\address{\parbox{\linewidth}{
		Department of Mathematics and Statistics, Ross Building, York University\\
		Toronto, Ontario, Canada}}
\address[Current address]{\parbox{\linewidth}{Department of Mathematics, University of Ulsan\\
		93 Daehak-ro, Nam-gu, Ulsan Metropolitan City, South Korea}}
\email{aanshu@yorku.ca}
\urladdr{https://sites.google.com/view/anshunirbhay/research?authuser=0}

\author{David Jekel}
\address{\parbox{\linewidth}{Department of
		Mathematical Sciences, University of Copenhagen \\
		Universitetsparken 5, 2100 Copenhagen \O, Denmark}}
\email{daj@math.ku.dk}
\urladdr{http://davidjekel.com}

\author{Therese Basa Landry}
\address{\parbox{\linewidth}{Department of
		Mathematics, South Hall, Room 6607, University of California, \\
		Santa Barbara, CA 93106, United States}}
\email{tlandry@ucsb.edu}


\begin{document}
	
	\maketitle
	
	\begin{abstract}
		We seek an analog for the quantum permutation group $S_n^+$ of the normalized Hamming distance for permutations.  We define three distances on the tracial state space of $C(S_n^+)$ that generalize the $L^1$-Wasserstein distance of probability measures on $S_n$ equipped with the normalized Hamming metric, for which we demonstrate basic metric properties, subadditivity under convolution, and density of the Lipschitz elements in the $\Cst$-algebra.
	\end{abstract}

	\section{Introduction}

	\subsection{Motivation}
	
	Our goal in this paper is to ``quantize'' the Hamming distance on the permutation group $S_n$ to obtain a ``quantum metric'' on $C(S_n^+)$.  This is motivated on the one hand by non-commutative geometry and Marc A. Rieffel's notion of quantum metric spaces \cite{Rieffel1998,Rieffel1999}, which was adapted to the setting of tracial states by Jacelon \cite{Jacelon2023metrics,Jacelon2024}, and on the other hand by recent developments in non-commutative optimal transport theory for quantum states, e.g.\ \cite{SkalskiTodorovTurowska2025,CaputoGerolinMoninaPortinale2024,ACJP2023}.
	
	The program of non-commutative geometry initiated by Alain Connes adapts classical tools from topology and Riemannian geometry to the operator algebraic setting.  For example, Connes showed in \cite{Connes1989,Connes1994} that the geodesic distance $d_g$ on a compact Riemannian spin manifold $M$ can be recovered from the $C^*$-algebra $C(M)$, the Hilbert space $H$ of $L^2$-spinor fields, and the Dirac operator.\footnote{Similarly, for a general oriented compact Riemannian manifold, the geodesic distance can be recovered from the Hodge Laplacian.}  He further noted that many of the arguments did not require commutativity of the underlying $\mathrm{C}^*$-algebra and hence should generalize to non-commutative operator algebras.  Rieffel developed a corresponding notion of metric spaces in the non-commutative setting \cite{Rieffel1998,Rieffel1999,Rieffel2002}.  The idea is to capture the underlying metric in a dual way through the associated $L^1$-Wasserstein distance on states.  Given a Dirac operator $D$ one can define a Lipschitz norm on the $\mathrm{C}^*$-algebra $A$ by $\norm{a}_{\Lip} = \norm{[a,D]}$, and then the ``$L^1$-Wasserstein distance'' on states is given by $W_1(\varphi,\psi) = \sup \{|\varphi(a) - \psi(a)|: \norm{a}_{\Lip} \leq 1\}$.  Rieffel investigated when this metric recovers the weak-$*$ topology on the state space \cite{Rieffel1998}.  In some cases, this metric also gives rise to a metric on the density space of a unital $\Cst$-algebra whose induced topology coincides with that of the Bures metric from quantum information theory \cite{AguilarBeheraOmlandWu}.
	
	For a discrete abelian group $G$, the dual group $\widehat{G}$ is compact and the (full or reduced) group $\Cst$-algebra $\Cst(G)$ is isomorphic to the continuous functions $C(\widehat{G})$. Thus, a natural class of examples to apply quantum metric theory in the sense of Rieffel arises from group $\mathrm{C}^*$-algebras.  In particular, Christ and Rieffel studied quantum metric space structures on $\Cst_r(G)$ where the metric is defined via a length function \cite{ChristRieffel2017}.  Length functions and the associated quantum metric spaces have also been investigated in the more general discrete \emph{quantum groups}, as well as the quantum permutation group $S_4^+$ \cite{BhomwickVoigtZacharias2015,AustadKyed2025}, and quantum metric structures have also been developed for homogeneous spaces of quantum $SU(2)$ \cite{AguilarKaad,AguilarKaadKyed1,AguilarKaadKyed2}.
	
	Compact quantum groups, as defined and developed by Woronowicz \cite{Woronowicz1979,Woronowicz1987,Woronowicz1998}, can be viewed as quantum analogues of the continuous functions on a group $C(G)$, where $G$ is a locally compact group.  To define multiplication in this setting, one must again dualize.  For a group $G$, there is a natural coassociative operation $\Delta: C(G) \to C(G \times G)$ given by $\Delta f(s,t) = f(st)$, called the \emph{comultiplication}, which encodes the group multiplication structure.  Thus, a quantum group structure on a $\Cst$-algebra $A$ is defined via a coassociative operation $\Delta: A \to A \otimes A$ and a counit $\varepsilon: A \otimes A \to A$ satisfying certain axioms.
	
	Our work focuses on the quantum permutation groups $S_n^+$.  In answer to a question posed by Connes, Wang developed the notion of  \emph{quantum symmetries} for an $n$-point set by defining \cite{Wang1998}).  The underlying $\Cst$-algebra $C(S_n^+)$ is the universal (unital) $\Cst$-algebra generated by elements $(u_{ij})_{i,j=1}^n$ such that $u_{ij}$ is a projection and $\sum_i u_{ij} = 1$ and $\sum_j u_{ij} =1 $.  The comultiplication on the generators is given by $\Delta(u_{ij}) = \sum_k u_{ik} \otimes u_{kj}$ which generalizes the matrix multiplication structure in a natural way; $u_{ij}$ corresponds in the classical case to the continuous function on $S_n$ that evaluates the $ij$ entry of the permutation matrix, or $\mathbbm{1}_{\sigma(i)=j}$.
	
	Since $C(S_n^+)$ is not a discrete quantum group for $n \geq 4$, it is not clear to us how to adapt the notion of length function per se \cite{AustadKyed2025}.  Rather, we take inspiration from cost operators used in quantum optimal transport theory.  One of the goals of this subject is to find appropriate analogs of Wasserstein distances on the space of quantum states of some system (often the space of states or density operators on a finite-dimensional $\mathrm{C}^*$-algebra).  Quantum Wasserstein distances can be used as a measure of the effort it takes to transform one state to another by elementary operations (see e.g.\ \cite{ACJP2023}), as well as to define a Riemannian structure analogous to the manifold of probability measures \cite{CarlenMaas2017}.
	
	As in classical optimal transport, one needs both a \emph{cost} to measure the distances between points as well as a notion of a \emph{transport plan} or a \emph{coupling} for joining two states.  Let us start with the cost operator.  If the underlying operator algebra $A$ is analogous to $C(X)$, then a cost or distance function $X \times X \to [0,\infty)$ would correspond to a positive element $C$ of $A \otimes A$.  The symmetry of distance means that $C$ is symmetric under the tensor flip, and the triangle inequality means that $\iota_{1,3}(C) \leq \iota_{1,2}(C) + \iota_{2,3}(C)$ where $\iota_{i,j}$ is the inclusion of $A \otimes A$ into the $i$th and $j$th copies in $A \otimes A \otimes A$.
	
	To construct such a cost operator on the quantum permutation group, recall that the normalized Hamming distance on the permutation group $S_n$ is given by
	\[
	d_H(\sigma,\sigma') = \frac{1}{n} |\{i: \sigma(i) \neq \sigma'(i)\}|.
	\]
	This can be expressed in terms of the entries of the permutation matrices as follows:
	\[
	d_H(\sigma,\sigma') = 1 - \frac{1}{n} |\{i: \sigma(i) = \sigma'(i)\}| = 1 - \frac{1}{n} |\{i: \sigma^{-1} \sigma'(i) = i\}| = 1 - \tr_n(\sigma^*\sigma') = 1 - \frac{1}{n} \sum_{i,j=1}^n \sigma_{ij} \sigma'_{ij}.
	\]
	This expression represents an element of $C(S_n \times S_n)$ since $\sigma \mapsto \sigma_{ij}$ is a continuous function on $S_n$.  To obtain the quantum version, we just have to replace this continuous function on $S_n$ with the element $u_{ij} \in C(S_n^+)$.  This leads to the following cost operator in $C(S_n^+) \otimes C(S_n^+)$:
	\[
	C_H = 1 - \frac{1}{n} \sum_{i,j=1}^n u_{ij} \otimes u_{ij}.
	\]
	We show that $C_H$ satisfies the triangle inequality (see \S \ref{subsec: tensor Hamming}).  We also show a subadditivity property under comultiplication (see \S \ref{subsec: tensor convolution}).  This is a natural property if we want to obtain metrics on the state space that interact with convolution operation on states of $C(S_n^+)$; for instance, in order to study convolution semigroups on $S_n^+$ \cite{FranzKulaSkalski2016}.
	
	The notion of coupling in the non-commutative setting is much more subtle.  A natural approach is to define a coupling of two states $\varphi_1$ and $\varphi_2$ on a $\mathrm{C}^*$-algebra $A$ as a state $\varphi$ on $A \otimes A$ with $\varphi(a \otimes 1) = \varphi_1(a)$ and $\varphi(1 \otimes a) = \varphi_2(a)$; see e.g.\ \cite{SkalskiTodorovTurowska2025,CaputoGerolinMoninaPortinale2024}. However, there are major difficulties in proving the triangle inequality because of the inability to amalgamate couplings, and sometimes the self-distance is nonzero.  One can also consider unital completely positive maps $A \to A$ transforming the one state into the other.  One thing that helps in the von Neumann algebraic setting is that by imposing appropriate modular conditions on the states, one can guarantee that couplings can be amalgamated via the relative tensor product \cite{Duvenhage2022}. However, it is unclear how to adapt this method to the setting of $\mathrm{C}^*$-algebras, especially with states that are not faithful.
	
	There are other approaches that are not based on a cost operator per se, but rather on other structures of the underlying space.  For instance, \cite{DePalmaMarvianTrevisanLloyd2021} defined a quantum Hamming distance on the state space for $n$ qdits, which gives a quantum version of quantifying changes one bit at a time.  Very recently Rieffel \cite{Rieffel2025} explained how similar quantum Hamming distances give rise to quantum compact metric spaces, and also shows how to encode the metric data in Dirac operator type structures.  In fact, another Hamming-like distance has been used to study the ideal structure of AF algebras \cite{AguilarBatterman}.  However, the Hamming distance in this paper is distinct and it is unclear how they relate, since the quantum permutation group does not relate naturally to an $n$-fold tensor product structure.
	
	
	Many of these issues become easier to deal with in the setting of \emph{traces} as opposed to general states. A trace on a $\Cst$-algebra automatically produces a tracial von Neumann algebra on which the trace is faithful (see Lemma \ref{lem: GNS tracial von Neumann algebra}), while in general it is difficult to tell whether a state on a $\Cst$-algebra gives rise to a faithful state on a von Neumann algebra and if so, how its modular group behaves.  The work of Biane and Voiculescu \cite{BianeVoiculescu2001} gives natural $L^p$-Wasserstein distances on tracial states on a certain universal $\Cst$-algebra with specified generators; here the couplings are given by embeddings into an arbitrary larger tracial von Neumann algebra, in which one can measure the distance between the two copies of the generators.  Moreover, the trace space of a $\Cst$-algebra, unlike the state space, is always a Choquet simplex, meaning that every trace has a unique decomposition in terms of extreme points (see \cite{BlackadarRordam2024}).
	
	Although restricting to tracial states may seem untenable for applications in quantum theory, many $\Cst$-algebras admit an abundance of traces, which have played a key role in the stably finite case of the classification of unital separable simple nuclear $\mathcal{Z}$-stable $\Cst$-algebras satisfying the UCT; see \cite{CGSTWclassification} for a version of the classification theorem and its history.  Motivated by the classification program, Jacelon \cite{Jacelon2023metrics,Jacelon2023chaotic,Jacelon2024} studied certain metrics on the trace space of classifiable $\Cst$-algebras $A$, showing that they are in duality with Lipschitz seminorms that vanish on the trace kernel of $A$, and hence are conceptually not far from quantum metrics defined by Rieffel.  In fact, he showed that any compact metric space can be realized as the extreme boundary of the trace space of some classifiable $\mathrm{C}^*$-algebra \cite[Theorem 4.4]{Jacelon2023chaotic}, \cite[Theorem A]{Jacelon2024}.
	
	While classifiable $\Cst$-algebras with large trace spaces require some effort to construct, it is much easier to see that many universal $\Cst$-algebras given by generators and relations have extremely large trace spaces.  For instance, the trace space on the universal $\mathrm{C}^*$-algebra generated by self-adjoints $\norm{x_j} \leq R$ need not be separable with respect to the Biane--Voiculescu--Wasserstein distance \cite[\S 5.5]{GangboJekelNamShlyakhtenko2022}. The $\Cst$-algebra of the quantum permutation group studied in this paper is also a universal $\Cst$-algebra that similarly has an abundance of traces, including many coming from finite-dimensional representations (see e.g.\ \cite{FreslonSkalskiWang2024}).

	\subsection{Results}
	
	In this work, we study three analogs of the Hamming distance on trace space of $C(S_n^+)$, which we call $\freedistance$, $\ldistance$, and $\tensordistance$.  The first two distances are based on the notion of tracial couplings as in Biane and Voiculescu's work \cite{BianeVoiculescu2001}, while the third is based on tracial tensor product couplings.
	
	\begin{definition}[Tracial couplings] \label{def: tracial coupling}
		Let $A$ be a $\mathrm{C}^*$-algebra and let $\varphi_1$, $\varphi_2 \in \mathcal{T}(A)$.  A \emph{tracial coupling} of $\varphi_1$ and $\varphi_2$ is a tuple $(M,\tau,\alpha_1,\alpha_2)$ where $(M,\tau)$ is a tracial von Neumann algebra and $\alpha_j: A \to M$ is a $*$-homomorphism such that $\varphi_j = \tau \circ \alpha_j$ for $j = 1$, $2$.
	\end{definition}
	
	For a tracial coupling $(M,\tau,\alpha_1,\alpha_2)$, the two projections $\alpha_1(u_{ij})$ and $\alpha_2(u_{ij})$ need not commute with each other, and $1 - (1/n) \sum_{i,j} \tau(\alpha_1(u_{ij}) \alpha_2(u_{ij}))$ need not define a metric.  However, there are at least two ways to modify this expression to obtain a metric.  The first is to replace the product of the two projections by the intersection $\alpha_1(u_{ij}) \wedge \alpha_2(u_{ij})$.
	
	\begin{definition}[Free Hamming distance via intersections]
		For $\varphi_1$, $\varphi_2 \in \mathcal{T}(C(S_n^+))$, define
		\[
		\freedistance(\varphi_1,\varphi_2) = \inf_{(M,\tau,\alpha_1,\alpha_2)} \left[ 1 - \frac{1}{n} \sum_{i,j=1}^n \tau(\alpha_1(u_{ij}) \wedge \alpha_2(u_{ij})) \right],
		\]
		where the infimum is over all tracial couplings $(M,\tau,\alpha_1,\alpha_2)$.
	\end{definition}
	
	As noted in \S \ref{subsec: free distance definition}, some tracial coupling always exists (e.g.\ by taking a tensor product), and although the class of tracial couplings is not a \emph{set}, we can rephrase the infimum in terms of sets in the situations that we care about, thanks to Lemma \ref{lemma:tracialcouplingequivalent}.  Hence, the infimum is well-defined and finite.
	
	Another natural choice of distance looks at the non-commutative $L^1$-distances $\norm{\alpha_1(u_{ij}) - \alpha_2(u_{ij})}_{L^1(M,\tau)}$.  To motivate this, recall that for two permutation matrices $\sigma$ and $\sigma'$, we have
	\[
	d_H(\sigma,\sigma') = \frac{1}{n} \sum_{i,j=1}^n \frac{1}{2} |\sigma_{ij} - \sigma_{ij}'|.
	\]
	Indeed, if $\sigma(j) = \sigma'(j)$, then the $j$th columns of the two matrices agree, but if $\sigma(j) \neq \sigma'(j)$, then the $j$th columns of the matrices differ in exactly two rows, so $\sum_i |\sigma_{ij} - \sigma_{ij}'| = 2$.  Since $u_{ij}$ is analogous to the value of the $ij$ entry of the matrix, we make the following definition.
	
	\begin{definition}[Free Hamming distance via $L^1$-norms] \label{def:l1distace}
		For $\varphi_1$, $\varphi_2 \in \mathcal{T}(C(S_n^+))$, define
		\[
		\ldistance(\varphi_1,\varphi_2) = \inf_{(M,\tau,\alpha_1,\alpha_2)} \frac{1}{n} \sum_{i,j=1}^n \frac{1}{2} \norm{\alpha_1(u_{ij}) - \alpha_2(u_{ij})}_{L^1(M,\tau)},
		\]
		where the infimum is over all tracial couplings $(M,\tau,\alpha_1,\alpha_2)$.
	\end{definition}
	
	We summarize our results about $\freedistance$ and $\ldistance$ in the following theorem.
	
	\begin{theorem} \label{thm: free distance summary} ~
		\begin{enumerate}[(1)]
			\item $\freedistance$ defines a metric on $\mathcal{T}(C(S_n^+))$.  (See Proposition \ref{prop: freedistance is a metric}.)
			\item For tracial states $\varphi_1$, $\varphi_2$, $\psi_1$, $\psi_2$, we have
			\[
			\freedistance(\varphi_1*\psi_1, \varphi_2*\psi_2) \leq \freedistance(\varphi_1,\varphi_2) + \freedistance(\psi_1,\psi_2).
			\]
			(See Proposition \ref{prop: convolution}.)
			\item The infimum in the definition of $\freedistance$ is a minimum (see Lemma \ref{lem: infimum achieved}.)
			\item $\freedistance$ is lower semi-continuous with respect to the weak-$*$ topology, and convergence in $\freedistance$ implies weak-$*$ convergence.  (See Lemma \ref{lem: weak* semicontinuity 1}.)
		\end{enumerate}
		The same properties also hold for $\ldistance$.  (See Lemma \ref{lem:inf-l1distance} and Proposition \ref{prop: L distance properties}.)
	\end{theorem}
	
	We remark that the proof of the triangle inequality uses the amalgamated free product construction, just as for Biane and Voiculescu's Wasserstein distance in free probability.  The traciality of $\varphi$ and $\psi$ is essential for the triangle inequality because we need to use the fact that for projections $p$ and $q$ in a tracial von Neumann algebra, $\tau(p \wedge q) + \tau(p \vee q) = \tau(p) + \tau(q)$, which fails for general states.
	
	The third notion of couplings that we consider is the na{\"\i}ve one of states on the tensor product with the given marginals, except that these states are restricted to be tracial.  The corresponding distance $\tensordistance$ satisfies the triangle inequality but the self-distance is not necessarily zero.  Our results on $\tensordistance$ are summarized as follows.
	
	\begin{theorem} \label{thm: tensor distance summary} ~
		\begin{enumerate}[(1)]
			\item $\tensordistance$ is nonnegative and symmetric and satisfies the triangle inequality.  (See Lemma \ref{lem: tensor metric properties}.)
			\item We have
			\[
			\tensordistance(\varphi_1,\varphi_2) \geq \frac{1}{2} \tensordistance(\varphi_1,\varphi_1) + \frac{1}{2} \tensordistance(\varphi_2,\varphi_2).
			\]
			(See Lemma \ref{lem: self distance estimate}.)
			\item For tracial states $\varphi_1$, $\varphi_2$, $\psi_1$, $\psi_2$, we have
			\[
			\tensordistance(\varphi_1*\psi_1, \varphi_2*\psi_2) \leq \tensordistance(\varphi_1,\varphi_2) + \tensordistance(\psi_1,\psi_2).
			\]
			(See Proposition \ref{prop: convolution for tensor distance}.)
		\end{enumerate}
	\end{theorem}
	
	A key point in proving the triangle inequality for the tensor couplings is the ability to amalgamate tensor couplings of $(\varphi_1,\varphi_2)$ and $(\varphi_2,\varphi_3)$.  This relies on the fact that the space of \emph{tracial states} is Choquet simplex, which fails for general states on a non-commutative $\mathrm{C}^*$-algebra.  This means that $\tensordistance$ can be expressed in terms of classical couplings of probability measures on the extreme boundary $\partial_e \mathcal{T}(C(S_n^+)))$, and the cost only depends on the values of $\varphi(u_{ij})$ for $\varphi$ in this extreme boundary.
	
	We relate the three distances to each other and to the classical Hamming distance as follows.
	
	\begin{proposition}  \label{prop: comparison summary} ~
		\begin{enumerate}[(1)]
			\item For traces $\varphi_1$, $\varphi_2 \in \mathcal{T}(C(S_n^+))$, we have
			\[
			\ldistance(\varphi_1,\varphi_2) \leq \freedistance(\varphi_1,\varphi_2) \leq \tensordistance(\varphi_1,\varphi_2).
			\]
			(See Proposition \ref{prop: comparison of distances}.)
			\item We furthermore have $\freedistance(\varphi_1,\varphi_2) \leq \frac{1}{2} \norm{\varphi_1 - \varphi_2}$.  (See Proposition \ref{prop: comparison TV}.)
			\item If $\varphi_1$ and $\varphi_2$ are both induced by classical probability measures $\mu_1$ and $\mu_2$ on $S_n$, then $\ldistance$, $\freedistance$, and $\tensordistance$ all agree with the classical $L^1$ Wasserstein distance of $\mu_1$ and $\mu_2$ with respect to $d_H$.  (See Proposition \ref{prop: recovery of classical distance}.)
			\item For $\varphi_1$, $\varphi_2 \in \mathcal{T}(C(S_n^+))$, each of the three distances is less than or equal to $1$ with equality if and only if $\sum_{i,j=1}^n \varphi_1(u_{ij}) \varphi_2(u_{ij}) = 0$.  (See Proposition \ref{prop: distance one}.)
		\end{enumerate}
	\end{proposition}
	
	Returning to the theme of quantum compact metric spaces, any distance $d$ on the trace space of a $\Cst$-algebra $A$ induces a Lipschitz seminorm on $A$ given by
	\[
	\norm{a}_{\Lip(d)} = \sup_{\varphi_1 \neq \varphi_2} \frac{|\varphi_1(a) - \varphi_2(a)|}{d(\varphi_1,\varphi_2)}.
	\]
	The domain of Lipschitz seminorm is defined as the set of $a$ where $\norm{a}_{\Lip(d)} < \infty$.
	
	\begin{proposition} \label{prop: dense Lipschitz}
		If $a$ is in the $*$-algebra generated by $\{u_{ij}: i, j \in [n]\}$ in $C(S_n^+)$, then $\norm{a}_{\Lip(\freedistance)} \leq \norm{a}_{\Lip(\ldistance)} < \infty$. So in particular, Lipschitz elements are dense for each of the two metrics. (See Proposition \ref{prop: Lipschitz dense}.)
	\end{proposition}
	
	Our work leaves open many basic questions, which we discuss in \S \ref{sec: questions}, including how the topologies induced by these metrics relate to each other, whether there is a reasonable way to extend the definitions to non-tracial states, and more.  In particular, it highlights the need for a better understanding of the trace space and the representation theory of $C(S_n^+)$.  See \ref{sec: questions} for further discussion.
	
	\subsection{Organization}
	
	The rest of the paper is organized as follows.
	\begin{itemize}
		\item \S \ref{sec: preliminaries} recalls background on $\Cst$-algebras, von Neumann algebras, and quantum groups.
		\item \S \ref{sec: free distance} constructs $\freedistance$ and proves Theorem \ref{thm: free distance summary}.
		\item \S \ref{sec: tensor distance} constructs $\tensordistance$ and proves Theorem \ref{thm: tensor distance summary}.
		\item \S \ref{subsec: comparison} relates the distances to each other and proves Proposition \ref{prop: comparison summary}.
		\item \S \ref{subsec: seminorms} proves Proposition \ref{prop: dense Lipschitz} on Lipschitz seminorms.
		\item \S \ref{sec: questions} gives questions for further research.
	\end{itemize}
	
	\subsection*{Acknowledgements}
	
	We thank the Fields Institute for their hospitality during the Thematic Program on Operator Algebras in Fall 2023, where this collaboration began.  We thank the Institute for Pure and Applied Mathematics for their hospitality during the long program on Noncommutative Optimal Transport in Spring 2025, including travel funding for A and DJ's visits.
	
	We thank Bhishan Jacelon for discussions of metrics on trace spaces and for suggesting to show density of the Lipschitz functions.  We thank Milad Marvian for suggesting to characterize when the distance is equal to $1$.  We thank Peixue Wu for references on quantum Wasserstein distances and Mikael R{\o}rdam for references on trace spaces.  We thank Anna Wysocza{\'n}ska-Kula and Uwe Franz for comments on the draft and discussion of recent work on convolution semigroups.
	
	\subsection*{Funding}
	
	DJ was partially supported by the Danish Independent Research Fund, grant 1026-00371B; and an EU Horizon Marie Sk{\l}odowska Curie Action, FREEINFOGEOM, grant id: 101209517. A was partially supported by the grant ``Logic and $\mathrm{C}^*$-algebras'' from the National Sciences and Engineering Research Council (Canada).
	
	Funded by the European Union.  Views and opinions expressed are those of the author(s) only and do not necessarily reflect 
	those of the European Union or the Research Executive Agency. Neither the European Union nor the granting authority can be held responsible for them.
	
	\section{Preliminaries} \label{sec: preliminaries}
	
	We assume basic familiarity with $\mathrm{C}^*$-algebras and von Neumann algebras; see e.g.\ \cite{Murphy1990,Davidson1996,Blackadar2006}.  Nonetheless, given that our paper spans several subdisciplines, we recall some basic notation and facts here for later use.
	
	\subsection{$\mathrm{C}^*$-algebras}
	
	\subsubsection{States and GNS construction}
	
	All $\Cst$-algebras in this papers are assumed to be unital.  The state space of a $\Cst$-algebra will be denoted $\mathcal{S}(A)$ and the space of tracial states will be denoted $\mathcal{T}(A)$.  These are both weak-$*$ convex and compact spaces.
	
	For a state $\varphi$, we denote by $(H_{\varphi},\pi_{\varphi})$ the \emph{GNS construction} (see e.g.\ \cite[\S 3.1]{Murphy1990}).  Recall the sesquilinear form $\ip{a,b}_\varphi = \varphi(a^*b)$ is nonnegative, and hence defines an inner product on $A / N_\varphi$, where $N_\varphi = \{a: \varphi(a^*a) = 0\}$.  The completion of $A / N_\varphi$ is denoted by $H_\varphi$.  For $a \in A$, we denote by $\widehat{a}$ the corresponding element of $H_\varphi$, i.e.\ the equivalence class $a + N_\varphi$ in $A / N_\varphi$.  Moreover, $\pi_\varphi: A \to B(H_\varphi)$ denotes the representation given by  $\pi_\varphi(x) \widehat{a} = \widehat{xa}$.
	
	In the course of the argument, we use the following facts relating states, $*$-homomorphisms, and GNS representations.
	
	\begin{lemma} \label{lem: quotient GNS isomorphism}
		Let $\theta: A \to B$ be a $*$-homomorphism, and let $\varphi$ be a state on $B$.
		\begin{enumerate}[(1)]
			\item Then $\varphi \circ \theta$ is a state on $A$.
			\item If $\varphi$ is tracial, then $\varphi \circ \theta$ is tracial.
			\item There is an isometric map $U: H_{\varphi \circ \theta} \to H_{\varphi}$ given by $U(\widehat{a}) = \widehat{\theta(a)}$, and for $x \in A$, we have $U \pi_{\varphi \circ \theta}(x) = \pi_{\varphi} \circ \theta(x) U$.
			\item If $\theta$ is surjective, then $U$ is surjective and there is an isomorphism $\pi_{\varphi \circ \theta}(A) \cong \pi_{\varphi}(B)$ given by $T \mapsto UTU^*$.
			\item If $\varphi$ is faithful, then there is an isomorphism $\pi_{\varphi \circ \theta}(A) \to \theta(A)$ given by $\pi_{\varphi \circ \theta}(x) \mapsto \theta(x)$.
		\end{enumerate}
	\end{lemma}
	
	\begin{proof}
		(1) and (2) are immediate.
		
		(3) Note that
		\[
		\ip{a,a'}_{\varphi \circ \theta} = \varphi \circ \theta(a^*a') = \varphi(\theta(a)^* \theta(a')) = \ip{\theta(a),\theta(a')}_\varphi.
		\]
		Hence, the mapping $\theta$ is isometric with respect to the Hilbert space norms, and so gives a well-defined and isometric operator $U: H_{\varphi \circ \theta} \to H_{\varphi}$.  Next, note that for $x, a \in A$, we have
		\[
		U \pi_{\varphi \circ \theta}(x) \widehat{a} = U \widehat{xa} = \widehat{\theta(xa)} = \widehat{\theta(x) \theta(a)} = \pi_{\varphi}(\theta(x)) \widehat{\theta(a)} = \pi_{\varphi}(\theta(x)) U \widehat{a}.
		\]
		Hence, $U \pi_{\varphi \circ \theta}(x) = \pi_{\varphi}(\theta(x)) U$ as desired.
		
		(4) If $\theta$ is surjective, the image of $U$ is dense in $H_{\varphi}$, so being an isometry $U$ is also surjective.  We then have $U \pi_{\varphi \circ \theta}(x) U^* = \pi_{\varphi}(\theta(x))$ which yields the asserted isomorphism.
		
		(5) In general, $\Ran(U)$ is the closure of $\widehat{\theta(a)}$ for $a \in A$ (the range again is closed because $U$ is an isometry).  Note that this is an invariant subspace with respect to $\pi_{\varphi}(\theta(A))$.  Therefore, we obtain a $*$-homomorphism $\alpha: \theta(A) \to B(\Ran(U))$ given by $\theta(x) \mapsto \theta(x) UU^* = UU^* \theta(x)$. The map $\alpha$ is injective because if $\alpha(y) = 0$, then $y UU^* \widehat{1} = y \widehat{1} = \widehat{y} = 0$, which implies that $y = 0$ by faithfulness of $\varphi$.  Hence, $\alpha^{-1}$ is a well-defined $*$-homomorphism.
		
		Now consider the map $\beta: \pi_{\varphi \circ \theta}(A) \to B(\Ran U)$ given by $\beta(T) = UTU^*$.  This is a $*$-homomorphism since $U^*U = 1$ and it is unital since $UU^*$ is the identity on $\Ran(U)$.  Moreover, $\beta$ is a $*$-isomorphism onto its image since $U^* \beta(T) U = T$.  By (1), we obtain
		\[
		\beta(\pi_{\varphi \circ \theta}(x)) = U \pi_{\varphi \circ \theta}(x) U^* = \pi_{\varphi}(\theta(x)) UU^* = \alpha(\theta(x)).
		\]
		Therefore, $\alpha^{-1} \circ \beta$ gives a $*$-isomorphism $\pi_{\varphi \circ \theta}(A) \to \theta(A)$.
	\end{proof}

	\subsubsection{Universal $\Cst$-algebras, free products, and tensor products} \label{subsec: C* tensor products}
	
	Next, we discuss universal $\Cst$-algebras given by generators and relations (see \cite{Blackadar1985}).  For a given set of generators $\mathcal{G} = \{ x_{\alpha}\}_{\alpha}$ and a set of relations $\mathcal{R}$ where the relation $r$ is allowed in $\mathcal{R}$ if $r$ is given as the following 
	\[ 
	\| p (x_{\alpha_1}, \hdots, x_{\alpha_n}, x_{\alpha_1}^*, \hdots ,x_{\alpha_n}^* )
	\| \leq \eta \]
	where $p \in \C \langle y_1, \hdots, y_{2n} \rangle$, the variables $y_i$'s do not commute and $\eta \geq 0$. 
	
	A representation $\phi$ of $(\mathcal{G}, \mathcal{R})$ is a set of operators $\{y_{\alpha}\}_{\alpha}$ in $B(H)$ for a Hilbert space $H$ such that $ \{y_{\alpha} = \phi(x_{\alpha}) \}_{\alpha} $ satisfy the relations $\mathcal{R}$. Consider the free $*$-algebra $\mathcal{F(G)}$ generated by $\mathcal{G}$ and $\mathcal{R}$.  The map $\phi$ induces a unique $*$-homomorphism
	\[
	\tilde{\phi} : \mathcal{F(G)} \to B(H).
	\]
	For a given set of generators and relations $(\mathcal{G},\mathcal{R})$, we can construct $\Cst(\mathcal{G,R})$ as long as it is \emph{admissible}, that is, if there is a representation of $(\mathcal{G}, \mathcal{R})$ and whenever $\{\phi_j, H_j\}$ is a family of representations of $(\mathcal{G}, \mathcal{R})$, then $\bigoplus_j {\phi_j}$ is a representation of $(\mathcal{G}, \mathcal{R})$ on $\bigoplus_j{ H_j}$.  For any $f \in \mathcal{F(G)}$, the norm is defined by $\|f\| : = \sup_j \{\|\phi_j(f) \| : \phi_j \text{ is a representation of } (\mathcal{G}, \mathcal{R})\}$.  The above two conditions imply that $\|f\| < \infty$ and is a $\Cst$-seminorm on $\mathcal{F(G)}$. Now the universal $\Cst$-algebra for $(\mathcal{G}, \mathcal{R})$, denoted by  $\Cst(\mathcal{G}, \mathcal{R})$, is defined as the completion of $\mathcal{F(G)}/\{f \in \mathcal{F(G)} : \|f \| = 0\}$.
	
	Given any two unital $\Cst$-algebras $A$ and $B$, the \emph{universal} or \emph{full free product} $A * B$ is the universal unital $\Cst$-algebra generated by copies of $A$ and $B$ \cite{Avitzour}.  It is constructed so that for any two representations $\pi_1$ and $\pi_2$ on a Hilbert space $H$, there is a unique representation of $\pi$ on $H$ extending $\pi_1$ and $\pi_2$.
	
	The \emph{maximal tensor product} $A \otimes_{\max} B$ of $\Cst$-algebras is the universal $\Cst$-algebra generated by commuting copies of $A$ and $B$ (see \cite[Theorem 6.3.5]{Murphy1990}).  It is constructed so that the images $A \otimes 1$ and $1 \otimes B$ commutes, and for any two representations $\pi_1$ and $\pi_2$ on a Hilbert space $H$ such that $\pi_1(A)$ and $\pi_2(B)$ commute, there is a unique representation of $\pi$ on $H$ extending $\pi_1$ and $\pi_2$. 
	
	The \emph{minimal tensor product} $A \otimes_{\min} B$ is the $\mathrm{C}^*$-algebra generated by the tensor product of given faithful representations of $A$ and $B$ respectively (see \cite[Theorem 6.3.3]{Murphy1990}).  Due to its common usage with quantum groups, we denote the minimal tensor product of $\mathrm{C}^*$-algebras simply by $\otimes$.
	
	\subsection{Von Neumann algebras}
	
	\subsubsection{Basic properties} \label{subsec: vN basic}
	
	A \emph{von Neumann algebra} is defined as a unital $*$-subalgebra $M \subseteq B(H)$ closed in the weak operator topology (WOT).  A state $\varphi$ on a von Neumann algebra $M$ is \emph{normal} if $\varphi$ is WOT-continuous on $(M)_1$.  A \emph{tracial von Neumann algebra} is a pair $(M,\tau)$ where $M$ is a von Neumann algebra and $\tau$ is a faithful, normal, tracial state.  We recall that any $\Cst$-algebra with a state gives rise to a von Neumann algebra via the GNS construction. The following statement is basic and well known but we were unable to find a specific reference, so we include the proof here.
	
	\begin{lemma} \label{lem: GNS tracial von Neumann algebra}
		Let $A$ be a $\mathrm{C}^*$-algebra, and let $\varphi \in \mathcal{T}(A)$.  Let $(H_\varphi,\pi_{\varphi})$ be the GNS construction of $(A, \varphi)$, let $M = \pi_{\varphi}(A)''$ be the von Neumann algebra generated by $\pi_{\varphi}(A)$, and let $\tau(x) = \ip{\widehat{1}, x \widehat{1}}$ for $x \in M$.  Then $(M,\tau)$ is a tracial von Neumann algebra.
	\end{lemma}
	
	\begin{proof}
		First, note that $\tau$ is a normal state on $M$ because it is a vector state.  Moreover, by virtue of the GNS construction, we have $\tau \circ \pi_{\varphi} = \varphi$. 
		In particular, $\tau$ is tracial on $\pi_{\varphi}(A)$.  Now if $x, y \in M$, fix nets $x_i$ and $y_j$ converging in the strong operator topology to $x$ and $y$.  Then
		\[
		\tau(xy) = \lim_{i} \tau(x_iy) = \lim_i \lim_j \tau(x_i y_j) = \lim_i \lim_j \tau(y_j x_i) = \lim_i \tau(yx_i) = \tau(yx).
		\]
		For the faithfulness of $\tau$, we use the right representation of $A$ given by $\rho(a) \widehat{b} = \widehat{ba}$.  To show that this is well-defined, note that
		\[
		\norm{\widehat{ba}}_{H_\varphi}^2 = \tau(a^*b^*ba) = \tau(baa^*b^*) \leq \norm{aa^*} \tau(bb^*) = \norm{a^*}^2 \tau(b^*b) = \norm{a}^2 \norm{\widehat{b}}_{H_\varphi}^2.
		\]
		Therefore, right multiplication by $a$ passes to a well-defined bounded operator on $L^2(H_\varphi)$.  It is easy to verify that the left multiplication by $a \in A$ and the right multiplication by $a' \in A$ commute. In particular, $\rho(A) \subseteq \pi_{\varphi}(A)'$, and so $\pi_{\varphi}(A)''$ commutes with $\rho(A)$.  Therefore, if $x \in \pi_{\varphi}(A)$ and $\tau(x^*x) = 0$, this means that $x \widehat{1} = 0$.  For $a \in A$, we have $0 = \rho(a) x \widehat{1} = x \rho(a) \widehat{1} = x \widehat{a}$.  Thus, $x$ is zero on a dense subset of $H_\varphi$, and so $x = 0$.  Therefore, $\tau$ is faithful.  Thus, $(M,\tau)$ is a tracial von Neumann algebra as desired.
	\end{proof}
	
	The GNS construction for a tracial von Neumann algebra $M \subseteq B(H)$ gives a standard representation which for most purposes can be used instead of the given representation on $H$.  We recall a few standard facts.  See e.g.\ \cite[\S 1.2]{JonesSunder1997}.  If $M \subseteq B(H)$ is a von Neumann algebra with faithful normal tracial state $\tau$, let $L^2(M,\tau) = H_\tau$ be the associated GNS space.  Then the GNS representation $\lambda: M \to B(L^2(M,\tau))$ is faithful, the image $\lambda(M)$ is a tracial von Neumann algebra, $\lambda$ defines a WOT homeomorphism from $M_1$ onto $\lambda(M)_1$.  There is also a right multiplication action $\rho(x)$ given by $\rho(x) \widehat{y} = \widehat{yx}$, which defines a $*$-homomorphism $M^{\operatorname{op}} \to B(L^2(M))$, where $M^{\operatorname{op}}$ is the algebra with reversed multiplication.  Moreover, $\lambda(M)$ and $\rho(M)$ commute, and in fact $\lambda(M)' = \rho(M)$ \cite[Theorem 1.2.4]{JonesSunder1997}.  In particular, $\lambda$ and $\rho$ produce a representation $\lambda \otimes \rho$ of $M \otimes_{\max} M^{\operatorname{op}}$ on $B(L^2(M,\tau))$.
	
	We remark that a tracial von Neumann algebra is determined up to isomorphism by the evaluation of the trace on $*$-polynomials in the generators.  A \emph{$*$-polynomial} in variables $(x_i)_{i \in I}$ is a linear combination of products of the terms $x_i$ and $x_i^*$, for instance, $x_1 x_2^* + 3 x_2^2 x_1^*x_3$.  The following is another folklore result that has not adequately been written down, so we include the proof here.
	
	\begin{lemma} \label{lem: law isomorphism}
		Let $(M,\tau)$ and $(\tilde{M},\tilde{\tau})$ be tracial von Neumann algebras.  Let $(x_i)_{i \in I}$ and $(\tilde{x}_i)_{i \in I}$ be elements of $M$ and $\tilde{M}$ respectively.  Suppose $N$ and $\tilde{N}$ be the von Neumann subalgebras generated by $(x_i)_{i \in I}$ and $(\tilde{x}_i)_{i \in I}$ respectively.  Suppose that for every $*$-polynomial, we have $\tau(p((x_i)_{i \in I})) = \tilde{\tau}(\tilde{p}((\tilde{x}_i)_{i \in I}))$.  Then there is a $*$-isomorphism $\pi: N \to \tilde{N}$ such that $\pi(x_i) = \tilde{x}_i$ and $\pi$ gives a WOT-homeomorphism $(N)_1$ to $(\tilde{N})_1$.
	\end{lemma}
	
	\begin{proof}
		Since the standard representation of $N$ on $L^2(N,\tau|_N)$ is faithful and gives the same WOT-topology on the unit ball, it suffices to prove the case where $M = N$ and $\tilde{M} = \tilde{N}$, and we may assume without loss of generality that $N$ and $\tilde{N}$ are subalgebras of $B(L^2(N,\tau))$ and $B(L^2(\tilde{N},\tilde{\tau})$ respectively.
		
		By the bicommutant theorem, $N$ is the SOT-closure of $p((x_i)_{i \in I})$ for $*$-polynomials $p$.  In particular, $\widehat{p((x_i)_{i \in I})}$ is dense in $L^2(N,\tau)$.  The analogous statement holds for $\tilde{N}$.  By assumption, we have $\tau[p((x_i)_{i \in I})^* q((x_i)_{i \in I})] = \tilde{\tau}[p((\tilde{x}_i)_{i \in I})^* q((\tilde{x}_i)_{i \in I})]$.  Therefore, there is a unitary isomorphism $U: L^2(N,\tau) \to L^2(\tilde{N},\tilde{\tau})$ satisfying $U[\widehat{p((x_i)_{i \in I})}] = U[\widehat{p((\tilde{x}_i)_{i\in I})}]$.  It is immediate that
		\[
		U \lambda(x_i) \widehat{p((x_i)_{i \in I})} = \lambda(\tilde{x}_i) U \widehat{p((x_i)_{i \in I})},
		\]
		and therefore $\lambda(\tilde{x}_i) = U\lambda(x_i) U^*$. Thus, the map $\operatorname{ad}_U: T \mapsto UTU^*$ sends the $*$-algebra generated by $(x_i)_{i \in I}$ to the $*$-algebra generated by $(\tilde{x}_i)_{i \in I}$.  The map $\operatorname{ad}_U$ is isometric with respect to operator norm and also a WOT-homeomorphism $B(L^2(N,\tau)) \to B(L^2(\tilde{N},\tilde{\tau}))$.  Hence, it maps the von Neumann algebra generated by $(x_i)_{i \in I}$ onto the von Neumann algebra generated by $(\tilde{x}_i)_{i \in I}$, or it maps $N$ onto $\tilde{N}$.
	\end{proof}
	
	In order to define $\ldistance$, we use the non-commutative $L^1$ space associated to a tracial von Neumann algebra.  We denote by $L^1(M,\tau)$ the completion of $M$ with respect to the norm $\norm{x}_{L^1(M,\tau)} = \tau(|x|)$ where $|x| = (x^*x)^{1/2}$.  There is a natural inclusion of $L^2(M,\tau)$ into $L^1(M,\tau)$.  In fact, one can define non-commutative $L^p$ spaces $L^p(M,\tau)$ for every $p \in [1,\infty]$, where $L^\infty(M,\tau)$ is $M$ itself.  These all reside within the algebra of affiliated operators and satisfy the non-commutative H{\"o}lder's inequality $\norm{xy}_{L^p(M,\tau)} \leq \norm{x}_{L^{p_1}(M,\tau)} \norm{y}_{L^{p_2}(M,\tau)}$ when $1/p = 1/p_1 + 1/p_2$.  However, we only need this for the values $1$, $2$, and $\infty$.  For background on non-commutative $L^p$ spaces, see e.g.\ \cite{PisierXu2003}.
	
	\subsubsection{Tensor products and amalgamated free products of von Neumann algebras}  \label{subsec: amalgamated free products}
	
	Given two tracial von Neumann algebras $(M,\tau)$ and $(N,\sigma)$, one defines their von Neumann algebraic tensor product $(M \overline{\otimes} N, \tau \otimes \sigma)$ as follows.  Consider $M \subseteq B(L^2(M,\tau))$ and $N \subseteq B(L^2(N,\sigma))$ using the standard left multiplication representation.  Then form the minimal $\mathrm{C}^*$ tensor product $M \otimes N$ on $L^2(M,\tau) \otimes L^2(N,\sigma)$, and let $M \overline{\otimes} N$ be the WOT-closure of $M \otimes N$ in $B(L^2(M,\tau) \otimes L^2(N,\sigma))$.  The vector $\widehat{1}_M \otimes \widehat{1}_N$ defines a normal tracial state on $M \otimes N$ which sends $a \otimes b$ to $\tau(a) \sigma(b)$.  Moreover, since $A \odot B$ is WOT-dense in $A \overline{\otimes} B$, there is only one normal state on $A \overline{\otimes} B$ that gives $\tau \otimes \sigma$ on $A \odot B$.  Faithfulness of $\tau \otimes \sigma$ on $M \overline{\otimes} N$ follows using the right multiplication action of $A \otimes B$ as in the proof of Lemma \ref{lem: GNS tracial von Neumann algebra}.
	
	As in \cite{BianeVoiculescu2001}, in order to show the triangle inequality for free Wasserstein distances, we need to be able to embed any two given tracial von Neumann algebras $(M_1,\tau_1)$ and $(M_2,\tau_2)$ into a larger one $(M,\tau)$, and to arrange that the embeddings of $M_1$ and $M_2$ agree on a common von Neumann subalgebra.  One construction that accomplishes this is the amalgamated free product.  We recall first the existence of trace-preserving conditional expectations.
	
	\begin{fact}[{See e.g.\ \cite[\S 3.1]{JonesSunder1997}}]
		Let $(M,\tau)$ be a tracial von Neumann algebra and let $N$ be a von Neumann subalgebra.  Then for each $x \in M$, there is a unique $E_N(x) \in N$ such that $\tau(xy) = \tau(E_N(x)y)$ for all $y \in N$.
		
		The map $E_N$ is called the \emph{trace-preserving conditional expectation} from $M$ to $N$.
	\end{fact}
	
	\begin{definition}[Free independence with amalgamation]
		Let $(M,\tau)$ be a tracial von Neumann algebra and $N$ a von Neumann subalgebra.  For $i \in I$, let $M_i$ be a von Neumann subalgebra with $N \subseteq M_i \subseteq M$.  We say that $(M_i)_{i \in I}$ are \emph{freely independent with amalgamation over $N$} if whenever $i_1 \neq i_2 \neq \cdots \neq i_k$ are indices in $I$ where consecutive indices are not equal, and $x_j \in M_{i_j}$ with $E_N[x_j] = 0$, then
		\[
		E_N[x_{i_1} \dots x_{i_k}] = 0.
		\]
	\end{definition}
	
	\begin{fact}[Amalgamated free product]
		Let $(M_i,\tau_i)_{i \in I}$ and $(N,\sigma)$ be given tracial von Neumann algebras, and let $\iota_i: N \to M_i$ be a trace-preserving $*$-homomorphism.  Then there exists a tracial von Neumann algebra $(M,\tau)$ and trace-preserving $*$-homomorphisms $\gamma_i: M_i \to M$ and $\gamma: N \to M$ such that
		\begin{itemize}
			\item $\gamma_i \circ \iota_i = \gamma$ for all $i \in I$.
			\item $M$ is generated by $\gamma_i(M_i)$ for $i \in I$.
			\item The subalgebras $\gamma_i(M_i)$ for $i \in I$ are freely independent with amalgamation over $\gamma(N)$.
		\end{itemize}
		Moreover, $M$ is unique up to isomorphism in the sense that if $(\tilde{M},\tilde{\gamma}_i,\tilde{\gamma})$ also satisfy this, then there is a unique isomorphism $\Phi: M \to \tilde{M}$ with $\Phi \circ \gamma_i = \tilde{\gamma}_i$.
	\end{fact}
	
	Note that the free product in this theorem is not a \emph{universal} free product but rather a \emph{reduced} free product, one constructed as acting on a particular Hilbert space rather than from all possible representations (there is also a universal amalgamated free product for $\mathrm{C}^*$-algebras).
	The construction of amalgamated (reduced) free product for $\mathrm{C}^*$-algebras was given in \cite{voiculescu1985}, and the version for tracial von Neumann algebras was given in \cite{Popa1993}.  For a detailed proof of this fact, see the construction of the $\mathrm{C}^*$-amalgamated free product in \cite[Theorem 4.7.2]{BrownOzawa2008}, apply \cite[Proposition 3.8.5]{DykemaNicaVoiculescu1992} to establish traciality of the state, and then extend the $\mathrm{C}^*$-algebra to a tracial von Neumann algebra by Lemma \ref{lem: GNS tracial von Neumann algebra}.  The uniqueness of the amalgamated free product of tracial von Neumann algebras follows from the $\mathrm{C}^*$-case in \cite[Theorem 4.7.2]{BrownOzawa2008} by extending to the von Neumann algebra along similar lines as Lemma \ref{lem: law isomorphism}.
	
	\subsubsection{Operations with projections in von Neumann algebras}
	
	The set of projections in a von Neumann algebra is closed under unions and intersections.  We recall the following elementary facts for later use.
	
	\begin{fact} \label{fact: stronglimitpqp}
		Let $p$ and $q$ be two projections in a von Neumann algebra $M$. 
		Then $(pqp)^k \to p \land q$ in SOT as $k \to \infty$.  If $\tau$ is a normal trace on $M$, then $\tau(p \wedge q) = \lim_{k \to \infty} \tau((pq)^k) = \inf_{k \in \N} \tau((pq)^k)$.
	\end{fact}
	
	\begin{proof}[Proof sketch]
		Since $(pq)^kp \geq (pq)^{k+1}p$, the sequence $(pq)^kp$ converges in SOT to some limit $r$, which is in $M$ since $M$ is SOT-closed.  By standard manipulations, one can show that $r \leq p$ and $r \leq q$ and $p \wedge q \leq r$.  If $\tau$ is a normal trace, then by SOT continuity $\tau((pq)^kp) \to \tau(p \wedge q)$, and by monotonicity, the limit over $k$ is the same as the infimum.
	\end{proof}

	\begin{fact} \label{fact: trace of max/min of projections}
		Let $(M,\tau)$ be a tracial von Neumann algebra.  If $p$ and $q$ are projections in $M$, then
		\[
		\tau(p \vee q) + \tau(p \wedge q) = \tau(p) + \tau(q).
		\]
	\end{fact}
	
	\begin{proof}[Proof sketch]
		As noted in \cite[III.1.1.3]{Blackadar2006}, letting $v$ be the partial isometry in the polar decomposition of $q(1-p)$, we have $v^*v = (p \vee q) - p$ and $vv^* = q - (p \wedge q)$ and therefore $(p \vee q) - p$ and $q - (p \wedge q)$ have the same trace.
	\end{proof}
	
	The next two facts are well-known and straightforward to check.
	
	\begin{fact} \label{fact: min versus sum}
		Let $(p_i)_{i=1}^n$ and $(q_i)_{i=1}^n$ be projections in some von Neumann algebra such that $p_i p_j = 0 = q_i q_j$ for $i \neq j$.  Then
		\[
		\left( \sum_{i=1}^n p_i \right) \wedge \left( \sum_{i=1}^n q_i \right) \geq \sum_{i=1}^n p_i \wedge q_i.
		\]
	\end{fact}
	
	\begin{fact} \label{fact: tensor versus min}
		Let $p_1, p_2$ be projections in some von Neumann algebra $M$ and let $q_1, q_2$ be projections in some von Neumann algebra $N$.  Then in $M \overline{\otimes} N$, we have $(p_1 \otimes q_1) \wedge (p_2 \otimes q_2) = (p_1 \wedge p_2) \otimes (q_1 \wedge q_2)$.
	\end{fact}
	
	\subsection{Compact Quantum groups}
	
	For a tensor product $A_1 \otimes \dots \otimes A_n$ and for distinct indices $j_1$, \dots, $j_k \in [n]$, we denote by $\iota_{j_1,\dots,j_k}$, the inclusion $A_{j_1} \otimes \dots \otimes A_{j_k} \to A_1 \otimes \dots \otimes A_n$ that maps $a_1 \otimes \dots \otimes a_k$ to the simple tensor with $a_i$ in the $j_i$'s component and $1$'s in the remaining components.  For instance, for a three-fold tensor product,
	\[
	\iota_{1,2}(a \otimes b) = a \otimes b \otimes 1, \qquad \iota_{3,2}(a \otimes b) = 1 \otimes b \otimes a.
	\]

	\begin{definition}[{Woronowicz \cite{Woronowicz1998}}]
		A \emph{compact quantum group} is a unital $\Cst$-algebra $A$ together with a $*$-homomorphism $\Delta : A \rightarrow A \otimes A$ such that:
		\begin{enumerate}
			\item $\Delta$ is coassociative i.e. $(\mathbbm{1}\otimes \Delta )\Delta = (\Delta \otimes \mathbbm{1})\Delta$.
			\item $ \overline{\textrm{span}(\Delta (A) (A \otimes \mathbbm{1}))} = \overline{\textrm{span}(\Delta (A) (\mathbbm{1} \otimes A))} = A \otimes A$.
		\end{enumerate} 
	\end{definition}

	\begin{remark}
		Any compact topological group $G$ is a compact quantum group as follows. The pair $(C(G), \Delta)$, where $\Delta : C(G) \to C(G) \otimes C(G)$ defined as $\Delta(f)(s,t) = f(st)$. 
	\end{remark}
	
	\begin{definition}[{\cite{Woronowicz1979}}]
		A \emph{Woronowicz Hopf $\Cst$-algebra }(compact matrix quantum group) is a unital $\Cst$-algebra $A$ together with a dense $*$-subalgebra $\mathcal{A}$ generated by $u_{ij}$ for some $n \in \mathbb{N} \text{ and } i,j \in \{ 1,...,n \} $, a unital $*$-homomorphism $\Delta: A \rightarrow A \otimes A$ such that, letting $U = [u_{ij}]_{i,j=1}^n \in M_n(A)$,
		\begin{enumerate}
			\item  $U^T$ is invertible in $M_n(A)$. 
			\item  $U$ is a unitary element of $M_n(A)$.
			\item  For every $i,j$, $ \Delta (u_{ij}) = \sum_{k=1}^n u_{ik} \otimes u_{kj}$. 
		\end{enumerate}
	\end{definition}
	
	The quantum permutation group $S_n^+$ is defined as follows. 
	The underlying $\Cst$-algebra $C(S_n^+)$ is the universal $\mathrm{C}^*$-algebra generated by $u_{ij}: 1 \leq i,j \leq n$ such that
	\[
	u_{ij} ^* = u_{ij} = u_{ij}^2, \qquad \sum_i u_{ij} = 1, \qquad \sum_j u_{ij} =1 
	\]
	Writing $\bM_n = \bM_n(\C)$ for $n \times n$ complex matrices, the comultiplication for $S_n^+$ is given by
	\[
	\Delta(u_{ij}) = \sum_{k=1}^n u_{ik} \otimes u_{kj}.
	\]
	$U$ will denote the matrix $U = [u_{ij}]_{i,j=1}^n \in \bM_n(C(S_n^+)) = \bM_n \otimes C(S_n^+)$.  The relations above are equivalent to $U^*U = UU^* = 1$ and each $u_{ij}$ being a projection. The comultiplication is given by
	\[
	(\id_{\bM_n} \otimes \Delta)(U) = \iota_{1,2}(U) \iota_{1,3}(U) \in \bM_n \otimes C(S_n^+) \otimes C(S_n^+).
	\]
	The quantum permutation group is the universal final object in the category $\mathcal{C}$ of quantum groups that act on an $n$-point space $X_n$; see \cite{Wang1998}
	
	For each $n \in \mathbb{N}$, the quantum permutation group $S_n^+$ contains the classical permutation group $S_n$ as a quantum subgroup, which means that there is a quotient map $C(S_n^+) \to C(S_n)$.
	
	\begin{theorem}[{\cite[Theorem 3.1]{Wang1998}}]
		\label{theorem:quotientmap}
		Let $\theta: C(S_n^+) \to C(S_n)$ be the surjective $*$-homomorphism given by the indicator function $\theta(u_{ij}) = \mathds{1}_{\sigma(i)=j}$. Then, $\theta$ is an embedding $S_n \subset S_n^+$ of quantum groups, that is, $\Delta_{C(S_n)} \circ \theta = (\theta \otimes \theta) \circ \Delta_{C(S_n^+)}$.
	\end{theorem}
	
	\begin{remark}
		For $n \leq 3$, $S_n$ and $S_n^+$ are isomorphic.  Otherwise, $S_n^+$ is neither classical nor finite-dimensional.
	\end{remark}
	
	
	Compact quantum groups have a natural analog of the Haar measure, known as the \emph{Haar state}.  Its existence and uniqueness for CQGs were proved by Woronowicz in the 1980s \cite{Woronowicz1987}. As it turns out, for the quantum permutation group $S_n^+$, the Haar state is also tracial and faithful, and it can be described combinatorially; see \cite{BanicaBichonCollins2007}.  Moreover, $C(S_n^+)$ has a large number of finite-dimensional representations and traces \cite{FreslonSkalskiWang2024}.  Other examples of quantum permutation matrices can be seen through the connection with quantum games \cite{LupiniMancinskaRoberson2020}.
	
	\section{Distance via free couplings} \label{sec: free distance}
	
	\subsection{Definition} \label{subsec: free distance definition}
	
	Recall (Definition \ref{def: tracial coupling}) that for a $\Cst$-algebra $A$ and $\varphi_1$, $\varphi_2 \in \mathcal{T}(A)$, a \emph{tracial coupling} is a tuple $(M,\tau,\alpha_1,\alpha_2)$ where $(M,\tau)$ is a tracial von Neumann algebra and $\alpha_j: A \to M$ is a $*$-homomorphism such that $\varphi_j = \tau \circ \alpha_j$ for $j = 1$, $2$.  Note that the set of tracial couplings is nonempty because we can take $M$ to be the von Neumann algebra generated by the GNS representation of $\varphi_1 \otimes \varphi_2$ on $A \otimes_{\min} A$ (see Lemma \ref{lem: GNS tracial von Neumann algebra}).  We next give an alternative description of tracial couplings in terms of traces on the universal free product $A * A$ that restrict to $\varphi_1$ and $\varphi_2$ on the first and second copies of $A$ respectively.  We remark that this is closely related to the bijection between factorizable quantum channels and traces on free products \cite{MusatRordam2021}.
	
	\begin{definition}\label{tracesfreeproduct}
		Let $A * A$ be the universal unital $\mathrm{C}^*$ free product (see \S \ref{subsec: C* tensor products}) and let $\iota_1$, $\iota_2: A \to A * A$ be the coordinate inclusions.  For $\varphi_1$, $\varphi_2 \in \mathcal{T}(A)$, let
		\[
		\mathcal{T}_{\varphi_1,\varphi_2}(A * A) = \{\varphi \in \mathcal{T}(A * A): \varphi \circ \iota_j = \varphi_j \text{ for } j = 1, 2 \}.
		\]
	\end{definition}
	
	\begin{lemma}
		\label{lemma:tracialcouplingequivalent}
		Let $A$ be a unital $\mathrm{C}^*$-algebra.    Let $\varphi_1, \varphi_2 \in \mathcal{T}(A)$.
		\begin{enumerate}[(1)]
			\item Given $\varphi \in \mathcal{T}_{\varphi_1,\varphi_2}(A * A)$, there exists a tracial coupling $(M,\tau,\alpha_1,\alpha_2)$ such that $\varphi = \tau \circ (\alpha_1 * \alpha_2)$.
			\item Conversely, for every coupling $(M,\tau,\alpha_1,\alpha_2)$, we have $\tau \circ (\alpha_1 * \alpha_2) \in \mathcal{T}_{\varphi_1,\varphi_2}(A * A)$.
			\item Let $(M,\tau,\alpha_1,\alpha_2)$ and $(\widetilde{M},\widetilde{\tau},\widetilde{\alpha}_1,\widetilde{\alpha}_2)$ be two tracial couplings of $\varphi_1, \varphi_2$ and let $\varphi$ and $\widetilde{\varphi}$ be the associated traces in $\mathcal{T}_{\varphi_1,\varphi_2}(A * A)$.  Suppose that $M = \mathrm{W}^*(\alpha_1(A),\alpha_2(A))$ and $\widetilde{M}= \mathrm{W}^*(\widetilde{\alpha}_1(A),\widetilde{\alpha}_2(A))$.  Then $\varphi = \widetilde{\varphi}$ if and only if there is a $*$-isomorphism $f: M \to \tilde{M}$ such that $\tau = \tilde{\tau} \circ f$ and $f \circ \alpha_j = \widetilde{\alpha}_j$ for $j = 1, 2$.
		\end{enumerate}
	\end{lemma}
	
	\begin{proof}
		(1) Let $\varphi \in \mathcal{T}(A * A)$.  Let $\pi_\varphi: A * A \to B(H_\varphi)$ be the GNS construction associated to $A * A$ and $\varphi$.  Let $M = \pi_\varphi(A * A)''$, which is a tracial von Neumann algebra by Lemma \ref{lem: GNS tracial von Neumann algebra}. Let $\alpha_j = \pi_\varphi \circ \iota_j$, and then $\tau \circ \alpha_j = \varphi \circ \iota_j = \varphi_j$, so $(M,\tau,\alpha_1,\alpha_2)$ is a tracial coupling.
		
		(2) Let $(M,\tau,\alpha_1,\alpha_2)$ be a tracial coupling.  Let $\alpha_1 * \alpha_2: A * A \to M$ be the $*$-homomorphism given by the universal property of the full free product, and let $\varphi := \tau \circ (\alpha_1 * \alpha_2)$.  Then $\varphi \circ \iota_j = \tau \circ \alpha_j = \varphi_j$ and hence $\varphi \in \mathcal{T}_{\varphi_{1}, \varphi_{2}} (A * A)$.
		
		(3) Suppose that $\varphi = \tilde{\varphi}$.  Note that $(\alpha_1(a))_{a \in A} \cup (\alpha_2(b))_{b \in A}$ is a generating set for $M$ and analogously $(\tilde{\alpha}_1(a))_{a \in A} \cup (\tilde{\alpha}_2(b))_{b \in A}$ is a generating set for $\tilde{M}$.  Since $\varphi = \tilde{\varphi}$, we have in particular that for $*$-polynomials
		\begin{align*}
			\tau(p((\alpha_1(a))_{a \in A},(\alpha_2(b))_{b \in B}))
			&= \varphi(p((\iota_1(a))_{a \in A},(\iota_2(b))_{b \in B}) \\
			&= \tilde{\varphi}(p((\iota_1(a))_{a \in A},(\iota_2(b))_{b \in B}) \\
			&= \tilde{\tau}(p((\tilde{\alpha}_1(a))_{a \in A},(\tilde{\alpha}_2(b))_{b \in B})).
		\end{align*}
		Therefore, by Lemma \ref{lem: law isomorphism}, we obtain the desired isomorphism $M \to \tilde{M}$ sending $\alpha_1(a)$ to $\tilde{\alpha}_1(a)$ and $\alpha_2(b)$ to $\tilde{\alpha}_2(b)$ for $a, b \in A$.  Conversely, if there is such a $*$-isomorphism $f$, then
		\begin{align*}
			\varphi(p((\iota_1(a))_{a \in A},(\iota_2(b))_{b \in B})
			&= \tau(p((\alpha_1(a))_{a \in A},(\alpha_2(b))_{b \in B})) \\ 
			&= \tilde{\tau}(p((\tilde{\alpha}_1(a))_{a \in A},(\tilde{\alpha}_2(b))_{b \in B})) \\
			&= \tilde{\varphi}(p((\iota_1(a))_{a \in A},(\iota_2(b))_{b \in B}).
		\end{align*}
		Thus, $\varphi$ and $\tilde{\varphi}$ agree on the $*$-subalgebra generated by $\iota_1(A)$ and $\iota_2(A)$, and therefore they agree on all of $A *_{\operatorname{full}} A$.
	\end{proof} 
	
	\begin{definition}[Quantum Hamming distance] \label{qhc}
		For $\varphi_1, \varphi_2 \in \mathcal{T}(C(S_n^+))$, let
		\begin{equation} \label{eq: qhc def}
			\freedistance(\varphi_1,\varphi_2) = \inf_{\text{tracial couplings } (M,\tau,\alpha_1,\alpha_2)} \left( 1 - \frac{1}{n} \sum_{i,j=1}^n \tau(\alpha_1(u_{ij}) \wedge \alpha_2(u_{ij}))  \right).
		\end{equation}
	\end{definition}

	Using Lemma \ref{lemma:tracialcouplingequivalent} and Fact \ref{fact: stronglimitpqp}, we can rewrite the quantum Hamming cost in terms of $\mathcal{T}_{\varphi_1,\varphi_2}(A * A)$.
	
	\begin{lemma} \label{lemma:hammingcostexpression} 
		For $j \in \{ 1, 2 \}$, let $\iota_j$ be as defined in \cref{tracesfreeproduct}. Given any $\varphi_1, \varphi_2 \in \mathcal{T}(C(S_n^+))$, the Hamming cost in \cref{qhc} is equivalent to the following:
		\[
		\freedistance(\varphi_1,\varphi_2) = \inf_{\varphi \in \mathcal{T}(A * A)} \sup_k \left( 1 - \frac{1}{n} \sum_{i,j=1}^n \varphi([\iota_1(u_{ij}) \iota_2(u_{ij})]^k)  \right).
		\]
	\end{lemma}
	
	\begin{proof}
		If  $\varphi_1, \varphi_2 \in \mathcal{T}(C(S_n^+))$, then
		\begin{align*}
			\freedistance(\varphi_1,\varphi_2) &= \inf_{\text{tracial couplings } (M,\tau,\alpha_1,\alpha_2)} \left( 1 - \frac{1}{n} \sum_{i,j=1}^n \tau(\alpha_1(u_{ij}) \wedge \alpha_2(u_{ij}))  \right) \\
			&= \inf_{\text{tracial couplings } (M,\tau,\alpha_1,\alpha_2)} \left( 1 - \frac{1}{n} \sum_{i,j=1}^n \tau(\lim_{k \to \infty}[\alpha_1(u_{ij}) \alpha_2(u_{ij})]^k)  \right)\\
			&= \inf_{\text{tracial couplings } (M,\tau,\alpha_1,\alpha_2)} \sup_k \left( 1 - \frac{1}{n} \sum_{i,j=1}^n \tau([\alpha_1(u_{ij}) \alpha_2(u_{ij})]^k)  \right)\\
			&\overset{\ref{lemma:tracialcouplingequivalent}}{=} \inf_{\varphi \in \mathcal{T}(A * A)} \sup_k \left( 1 - \frac{1}{n} \sum_{i,j=1}^n \varphi([\iota_1(u_{ij}) \iota_2(u_{ij})]^k)  \right).
		\end{align*}
	\end{proof}
	
	Next, we note that the infimum in the definition of the Hamming cost is always achieved; this is analogous to \cite[Proposition 1.4]{BianeVoiculescu2001} for the Wasserstein distance in free probability.
	
	\begin{lemma} \label{lem: infimum achieved}
		Let $A = C(S_n^+)$, and let $\varphi_1$, $\varphi_2 \in \mathcal{T}(A)$.  Then there exists a tracial coupling $(M,\tau,\alpha_1,\alpha_2)$ that achieves the infimum in \eqref{eq: qhc def}.
	\end{lemma}
	
	\begin{proof}
		Define $f_k: \mathcal{T}_{\varphi_1,\varphi_2}(A * A)$ by
		\[
		f_k(\varphi) = 1 - \frac{1}{n} \sum_{i,j=1}^n \varphi[(\iota_1(u_{ij}) \iota_2(u_{ij}))^k].
		\]
		Let $f(\varphi) = \inf_{k \in \N} f_k(\varphi)$.  Since $f_k$ is weak-$*$ continuous on $\mathcal{T}_{\varphi_1,\varphi_2}(A * A)$, then $f$ is an infimum of lower semi-continuous functions and hence lower semi-continuous.  Note that $\mathcal{T}_{\varphi_1,\varphi_2}(A * A)$ is a closed subset of the compact set $\mathcal{T}(A * A)$, hence compact.  Therefore, $f$ achieves an infimum on $\mathcal{T}_{\varphi_1,\varphi_2}(A * A)$ at some point $\varphi$.  Let $(M,\tau,\alpha_1,\alpha_2)$ be the associated tracial coupling by Lemma \ref{lemma:tracialcouplingequivalent}.  Then, in light of Lemma \ref{lemma:hammingcostexpression}, $(M,\tau,\alpha_1,\alpha_2)$ achieves the infimum in \eqref{eq: qhc def}.
	\end{proof}
	
	\subsection{Metric properties}
	
	\begin{proposition}\label{prop: freedistance is a metric}
		$\freedistance$ defines a metric on $\mathcal{T}(C(S_n^+))$.
	\end{proposition}
	
	\begin{proof}
		
		\textbf{Nonnegativity:}  Let $\varphi_1, \varphi_2$ be traces in $\mathcal{T}(C(S_n^+))$, and let $(M,\tau,\alpha_1,\alpha_2)$ be a tracial coupling.  Then
		\begin{align*} 
			& 1- \frac{1}{n} \sum_{i,j=1}^n \tau(\alpha_1(u_{ij}) \wedge \alpha_2(u_ij)) \\
			& =  \sum_{i,j=1}^n \frac{1}{n}(\tau(\alpha_1(u_{ij})) -  \tau(\alpha_1(u_{ij}) \wedge \alpha_2(u_ij))) \geq 0
		\end{align*}
		because $\alpha_1(u_{ij}) \geq (\alpha_1(u_{ij}) \wedge \alpha_2(u_{ij}))$ for $i, j = 1, \dots, n$.  Hence, taking the infimum over tracial couplings, we have $\freedistance(\varphi_1,\varphi_2) \geq 0$.
		
		\textbf{Nondegeneracy:}  We show $\freedistance(\varphi_1,\varphi_2) = 0$ if and only if $\varphi_1 = \varphi_2$.  If $\varphi_1 = \varphi_2$, then we can take $(M,\tau)$ to be the von Neumann algebra generated by the GNS construction of $(C(S_n^+),\varphi_1)$ as in Lemma \ref{lem: GNS tracial von Neumann algebra}, and let $\alpha_1 = \alpha_2$ be the GNS representation itself.  Then $\alpha_1(u_{ij}) \wedge \alpha_2(u_{ij}) = \alpha_1(u_{ij})$, and so
		\[
		1 - \frac{1}{n} \sum_{i,j=1}^n \tau(\alpha_1(u_{ij}) \wedge \alpha_2(u_{ij})) = 1 - \frac{1}{n} \sum_{i,j=1}^n \tau(\alpha_1(u_{ij})) = 1 - \frac{1}{n} \sum_{i=1}^n 1 = 0,
		\]
		so $\freedistance(\varphi_1,\varphi_1) = 0$.
		
		Conversely, suppose that $\freedistance(\varphi_1,\varphi_2) = 0$. By Lemma \ref{lem: infimum achieved}, there is a tracial coupling $(M,\tau, \alpha_1, \alpha_2)$ such that
		\[
		0 = \freedistance(\varphi_1,\varphi_2) = 1 - \frac{1}{n} \sum_{i,j =1}^n \tau(\alpha_1(u_{ij}) \wedge \alpha_2(u_{ij})) =  \frac{1}{n}  \sum_{i,j=1}^n (\tau(\alpha_1(u_{ij})) - \tau(\alpha_1(u_{ij}) \wedge \alpha_2(u_ij))).
		\]
		Since we are adding nonnegative elements, we get $\tau(\alpha_1(u_{ij}) - \alpha_1(u_{ij}) \wedge \alpha_2(u_{ij})) = 0$ for all  $i,j$. Thus, faithfulness of $\tau$ implies that $\alpha_1(u_{ij}) = \alpha_1(u_{ij}) \wedge \alpha_2(u_{ij})$ for all $i,j$.
		Symmetrically, \( \alpha_2(u_{ij}) = \alpha_1(u_{ij}) \wedge \alpha_2(u_{ij})\), and therefore $\alpha_1(u_{ij}) = \alpha_2 (u_{ij})\) for all $i,j$.  Since $\{u_{ij}\}_{i,j}$ generate $C(S_n^+)$, we have $\alpha_1 = \alpha_2$, hence $\varphi_1 = \tau \circ \alpha_1 = \tau \circ \alpha_2 = \varphi_2$.
		
		\textbf{Symmetry:} Observe that
		$(M,\tau,\alpha_1,\alpha_2)$ is a coupling of $(\varphi_1,\varphi_2)$ if and only if $(M,\tau,\alpha_2,\alpha_1)$ is a coupling of $(\varphi_2,\varphi_1)$ and $\alpha_1(u_{ij}) \wedge \alpha_2(u_{ij}) = \alpha_2(u_{ij}) \wedge \alpha_1(u_{ij})$. Therefore,
		\begin{align*}
			\freedistance(\varphi_1,\varphi_2) & = \inf_{\text{tracial couplings } (M,\tau,\alpha_1,\alpha_2)} \left( 1 - \frac{1}{n} \sum_{i,j=1}^n \tau(\alpha_1(u_{ij}) \wedge \alpha_2(u_{ij})) \right) \\
			& = \inf_{\text{tracial couplings } (M,\tau,\alpha_2,\alpha_1)} \left( 1 - \frac{1}{n} \sum_{i,j=1}^n \tau(\alpha_2(u_{ij}) \wedge \alpha_1(u_{ij})) \right)\\
			& = \freedistance(\varphi_2, \varphi_1).
		\end{align*}
		
		\textbf{Triangle inequality:}  Let $\varphi_1, \varphi_2, \varphi_3 \in \mathcal{T}(C(S_n^+))$ and consider couplings $(M_1, \tau_1, \alpha_1, \alpha_2), (M_2, \tau_2, \beta_1, \beta_2)$.  Since $\tau_1 \circ \alpha_2 = \varphi_2 = \tau_2 \circ \beta_1$, the von Neumann subalgebras generated of $M_1$ and $M_2$ generated respectively by $\alpha_2(C(S_n^+))$ and $\beta_1(C(S_n^+))$ are $*$-isomorphic in a trace-preserving way by Lemma \ref{lem: law isomorphism}. Therefore, there is a free product $(M,\tau)$ of $(M_1,\tau_1)$ and $(M_2,\tau_2)$ with amalgamation over $\alpha_2(C(S_n^+)) \cong \beta_1(C(S_n^+))$.  Let $\eta_j$ be the inclusion of $M_j$ into $M$.  Let
		\begin{align*}
			\gamma_1 &= \eta_1 \circ \alpha_1: C(S_n^+) \to M \\
			\gamma_2 &= \eta_1 \circ \alpha_2 = \eta_2 \circ \beta_1: C(S_n^+) \to M \\
			\gamma_3 &= \eta_2 \circ \beta_2: C(S_n^+) \to M
		\end{align*}
		Then $\tau \circ \gamma_1 = \tau \circ \eta_1 \circ \alpha_1 = \tau_1 \circ \alpha_1 = \varphi_1$.  Similarly, $\tau \circ \gamma_2 = \tau \circ \eta_2 \circ \beta_2 = \tau_2 \circ \beta_2 = \varphi_3$.  Therefore, $(M,\tau,\gamma_1,\gamma_3)$ is a tracial coupling of $(\varphi_1,\varphi_3)$.
		
		For each $i,j = 1, \dots n$ and $k = 1, 2, 3$, the element $\gamma_k(u_{ij})$ is a projection, and therefore their maximum and minimum are also projections. By Fact \ref{fact: trace of max/min of projections}, for each $i,j$, we have
		\begin{equation}\label{rel-2couplings}
			\begin{split}
				\tau(\gamma_1(u_{ij}) \wedge \gamma_2(u_{ij})) + \tau(\gamma_2(u_{ij}) \wedge \gamma_3(u_{ij})) & = \tau(\gamma_1(u_{ij}) \wedge \gamma_2(u_{ij})) \wedge ( \gamma_2(u_{ij}) \wedge \gamma_3(u_{ij}))  \\
				& + \tau((\gamma_1(u_{ij}) \wedge \gamma_2(u_{ij})) \vee  (\gamma_2(u_{ij}) \wedge \gamma_3(u_{ij})) \\
				& \leq \tau(\gamma_1(u_{ij}) \wedge \gamma_3(u_{ij})) + \tau(\gamma_2(u_{ij}))
			\end{split}
		\end{equation}
		since 
		$\gamma_1(u_{ij}) \wedge \gamma_2(u_{ij}) \wedge \gamma_3(u_{ij}) \leq \gamma_1(u_{ij}) \wedge \gamma_3(u_{ij})$ and $ (\gamma_1(u_{ij}) \wedge \gamma_2(u_{ij})) \vee (\gamma_2(u_{ij}) \wedge \gamma_3(u_{ij}) )\leq \gamma_2(u_{ij})$
		and $\tau$ is positive.  Therefore, by \eqref{rel-2couplings}, 
		\[
		- \frac{1}{n} \sum_{i,j=1}^n \tau(\gamma_1(u_{ij}) \wedge \gamma_3(u_{ij})) - \frac{1}{n} \sum_{i,j=1}^n \tau(\gamma_2(u_{ij})) \leq - \frac{1}{n} \sum_{i,j=1}^n \tau(\gamma_1(u_{ij}) \wedge \gamma_2(u_{ij})) - \frac{1}{n} \sum_{i,j=1}^n \tau(\gamma_2(u_{ij}) \wedge \gamma_3(u_{ij})).
		\]
		By adding 1 both sides and using $\sum_{i,j=1}^n \tau(\gamma_2(u_{ij})) = 1$, we get 
		\begin{align}\label{rel-rewriting}
			1 - \frac{1}{n} \sum_{i,j=1}^n \tau(\gamma_1(u_{ij}) \wedge \gamma_3(u_{ij})) &\leq 1 - \frac{1}{n} \sum_{i,j=1}^n \tau(\gamma_1(u_{ij}) \wedge \gamma_2(u_{ij})) + 1 - \frac{1}{n}\sum_{i,j=1}^n \tau(\gamma_2(u_{ij}) \wedge \gamma_3(u_{ij})) \\
			&= 1 - \frac{1}{n} \sum_{i,j=1}^n \tau_1(\alpha_1(u_{ij}) \wedge \alpha_2(u_{ij})) + 1 - \frac{1}{n}\sum_{i,j=1}^n \tau_2(\beta_1(u_{ij}) \wedge \beta_2(u_{ij})),
		\end{align}
		where in the last step we used that the embeddings of $M_1$ and $M_2$ into $M$ are trace-preserving.  Since the tracial couplings of $(\varphi_1,\varphi_2)$ and $(\varphi_2,\varphi_3)$ were arbitrary, \cref{lemma:tracialcouplingequivalent} and taking the infimum yields
		\[
		\freedistance(\varphi_1,\varphi_3) \leq \freedistance(\varphi_1,\varphi_2) + \freedistance(\varphi_2,\varphi_3). \qedhere
		\]
	\end{proof}
	
	We next show the properties of $\freedistance$ with respect to the weak-$*$ topology.
	
	\begin{lemma} \label{lem: weak* semicontinuity 1}
		$\freedistance$ is a weak-$*$ lower semicontinuous function on $\mathcal{T}(C(S_n^+)) \times \mathcal{T}(C(S_n^+))$.  Moreover, if $\freedistance(\varphi_k,\varphi) \to 0$, then $\varphi_k \to \varphi$ in the weak-$*$ topology.
	\end{lemma}
	
	\begin{proof}
		Since $\mathcal{T}(C(S_n^+))$ is compact and metrizable, it suffices to show that, given sequences $\varphi_{1,m} \to \varphi_1$ and $\varphi_{2,m} \to \varphi_2$, we have
		\[
		\freedistance(\varphi_1,\varphi_2) \leq \liminf_{m \to \infty} \freedistance(\varphi_{1,m}, \varphi_{2,m}).
		\]
		Let $f$ be as in the proof of Lemma \ref{lem: infimum achieved}, and let $\psi_m \in \mathcal{T}_{\varphi_{1,m},\varphi_{2,m}}$ achieve the infimum of $f$ over $\mathcal{T}_{\varphi_{1,m},\varphi_{2,m}}(A * A)$.  Since $\mathcal{T}(A * A)$ is compact, there exists a subsequence $\psi_{m_\ell}$ that weak-$*$ converges to some $\psi \in \mathcal{T}(A * A)$.  Note that
		\[
		\psi \circ \iota_j = \lim_{\ell \to \infty} \psi_{m_\ell} \circ \iota_j = \lim_{\ell \to \infty} \varphi_{j,m_\ell} = \varphi_j \text{ for } j = 1, 2,
		\]
		hence $\psi \in \mathcal{T}_{\varphi_1,\varphi_2}(A * A)$.  Then
		\[
		\freedistance(\varphi_1,\varphi_2) \leq f(\psi) \leq \liminf_{\ell \to \infty} f(\psi_{m_\ell}) \leq \liminf_{m \to \infty} f(\psi_m) = \liminf_{m \to \infty} \freedistance(\varphi_{1,m},\varphi_{2,m}),
		\]
		which completes the proof of lower semi-continuity.
		
		For the second claim, suppose $\varphi_k \to \varphi$ with respect to $\freedistance$.  Since $\mathcal{T}(C(S_n^+))$ is compact and metrizable, it suffices to show that for every subsequence of $\varphi_k$ has a further subsequence that converges to $\varphi$.  If $\varphi_{k_\ell}$ is a subsequence, pass to a subsequence that converges to some $\psi$.  By lower semi-continuity, $\freedistance(\psi,\varphi) \leq \liminf_{\ell \to \infty} \freedistance(\varphi_{k_\ell},\varphi) = 0$, and hence $\psi = \varphi$ by non-degeneracy of the metric.
	\end{proof}

	\subsection{Behavior under convolution}
	
	Recall that the comultiplication map $\Delta$ on $C(S_n^+)$, is given by  $\Delta(u_{ij}):= \sum_{k=1}^n u_{ik} \otimes u_{kj}$ and the convolution of two states $\tau, \sigma \in \mathcal{T}(C(S_n^+)) \in \mathcal{T}(M \otimes N)$.
	\[
	\tau * \sigma = (\tau \otimes \sigma) \circ \Delta.
	\]
	
	\begin{proposition} \label{prop: convolution}
		Let $\varphi_1$, $\varphi_2$, $\psi_1$, $\psi_2  \in \mathcal{T}(C(S_n^+))$.  Then
		\[
		\freedistance(\varphi_1 * \psi_1, \varphi_2 * \psi_2) \leq \freedistance(\varphi_1,\varphi_2) + \freedistance(\psi_1,\psi_2).
		\]
	\end{proposition}
	
	\begin{proof}
		Consider tracial couplings $(M, \tau, \alpha_1, \alpha_2)$ for $(\varphi_1,\varphi_2)$ and $(N, \sigma, \beta_1, \beta_2)$ for $(\psi_1,\psi_2)$.  Consider the tracial von Neumann algebra $(M \overline{\otimes} N, \tau \otimes \sigma)$, and define $\gamma_1, \gamma_2: C(S_n^+) \to M \overline{\otimes} N$ by $\gamma_j := (\alpha_j \otimes \beta_j) \circ \Delta$.  For $j = 1, 2$, we have
		\begin{align*}
			(\tau \otimes \sigma) \circ \gamma_j & = (\tau \otimes \sigma) \circ (\alpha_j \otimes \beta_j) \circ \Delta \\
			& = ((\tau \circ \alpha_j) \otimes (\sigma \circ \beta_j)) \circ \Delta \\
			& = (\varphi_j \otimes \psi_j) \circ \Delta\\
			& = \varphi_j * \psi_j.
		\end{align*}
		Therefore, $(N \overline{\otimes} M, \tau \otimes \sigma, \gamma_1, \gamma_2)$ is a tracial coupling for $\varphi_1 * \psi_1$ and $\varphi_2 * \psi_2$. For any $i,j \in [n]$, note that $\alpha_1(u_{ik}) \otimes \beta_1(u_{kj})$ and $\alpha_1(u_{ik'}) \otimes \beta_1(u_{k'j})$ are orthogonal for $k \neq k'$ (and similarly with $\alpha_2$ and $\beta_2$ instead of $\alpha_1$ and $\beta_1$), and so we compute using Facts \ref{fact: min versus sum} and \ref{fact: tensor versus min} that
		\begin{align*}
			\gamma_1(u_{ij}) \wedge \gamma_2(u_{ij}) & =  (\alpha_1 \otimes \beta_1)\Delta(u_{ij}) \wedge (\alpha_2 \otimes \beta_2)\Delta(u_{ij})\\
			& =  (\alpha_1 \otimes \beta_1) \left(\sum_{k=1}^n u_{ik} \otimes u_{kj} \right) \wedge (\alpha_2 \otimes \beta_2)\left(\sum_{k=1}^n u_{ik} \otimes u_{kj}\right)\\
			& = \left( \sum_{k=1}^n (\alpha_1(u_{ik}) \otimes \beta_1(u_{kj}) ) \right) \wedge \left( \sum_{k=1}^n (\alpha_2(u_{ik}) \otimes \beta_2(u_{kj})) \right) \\
			& \geq \sum_{k=1}^n (\alpha_1(u_{ik}) \otimes \beta_1(u_{ik})) \wedge (\alpha_2(u_{ik})) \otimes \beta_2(u_{kj})) \\
			&= \sum_{k=1}^n (\alpha_1(u_{ik}) \wedge \alpha_2(u_{ik})) \otimes (\beta_1(u_{kj}) \wedge \beta_2(u_{kj})).
		\end{align*}
		By positivity of $\tau \otimes \sigma$, 
		\begin{align*}
			(\tau \otimes \sigma)(\gamma_1 (u_{ij}) \wedge \gamma_2(u_{ij})) & \geq \sum_{k=1}^n (\tau \otimes \sigma) ((\alpha_1(u_{ik}) \wedge \alpha_2(u_{ik})) \otimes (\beta_1(u_{kj}) \wedge \beta_2(u_{kj}))\\
			& = \sum_{k=1}^n \tau (\alpha_1(u_{ik}) \wedge \alpha_2(u_{ik})) \sigma(\beta_1(u_{kj}) \wedge \beta_2(u_{kj}))
		\end{align*}
		Therefore, 
		\begin{align*}
			& 1 - \frac{1}{n} \sum_{i,j=1}^n (\tau \otimes \sigma)(\gamma_1 (u_{ij}) \wedge \gamma_2(u_{ij})) \\
			& \leq 1 - \frac{1}{n} \sum_{i,k=1}^n  \tau (\alpha_1(u_{ik}) \wedge \alpha_2(u_{ik})) + \frac{1}{n} \sum_{i,k=1}^n \tau (\alpha_1(u_{ik}) \wedge \alpha_2(u_{ik})) \\
			& \quad -\frac{1}{n} \sum_{i,j,k=1}^n \tau (\alpha_1(u_{ik}) \wedge \alpha_2(u_{ik})) \sigma(\beta_1(u_{kj}) \wedge \beta_2(u_{kj})) \\
			& \leq 1 - \frac{1}{n} \sum_{i,k=1}^n  \tau (\alpha_1(u_{ik}) \wedge \alpha_2(u_{ik})) + \frac{1}{n} \sum_{i,k=1}^n \tau (\alpha_1(u_{ik}) \wedge \alpha_2(u_{ik})) (1 -  \sum_{j=1}^n \sigma(\beta_1(u_{kj} ) \wedge \beta_2(u_{kj})) \\
			& \leq 1 - \frac{1}{n} \sum_{i,k=1}^n  \tau (\alpha_1(u_{ik}) \wedge \alpha_2(u_{ik}))  + 1 - \frac{1}{n} \sum_{k, j=1}^n \sigma(\beta_1(u_{kj} ) \wedge \beta_2(u_{kj})).
		\end{align*}
		Since the couplings of $(\varphi_1,\varphi_2)$ and $(\psi_1,\psi_2)$ were arbitrary, we get the required inequality. The small thing to note here is the identification for the state as mentioned in the paragraph before \cref{prop: convolution}.
	\end{proof}

	\subsection{Distance via $L^1$-norm}
	
	Recall the second distance $\ldistance$ from Definition
	\ref{def:l1distace} given by 
	\[
	\ldistance(\varphi_1, \varphi_2) = \inf_{(M,\tau, \alpha_1, \alpha_2)}  \frac{1}{n} \sum_{i,j=1}^n \frac{1}{2} \left \| \alpha_1(u_{ij}) - \alpha_2(u_{ij}) \right \|_{L^1(M, \tau)},
	\]
	where $L^1(M,\tau)$ is the non-commutative $L^1$ space (see \S \ref{subsec: vN basic}).
	
	\begin{lemma}\label{lem:inf-l1distance}
		The infimum in \cref{def:l1distace} is achieved. 
	\end{lemma}
	
	\begin{proof}
		Let $C_1$ be the element of the universal free product $A * A$ given by
		\[
		C_1 = \frac{1}{2n} \sum_{i,j=1}^n |\iota_1(u_{ij}) - \iota_2(u_{ij})|,
		\]
		where $|x| = (x^*x)^{1/2}$.  If $(M,\tau,\alpha_1,\alpha_2)$ is a tracial coupling and $\varphi \in \mathcal{T}_{\varphi_1,\varphi_2}(A * A)$ is the corresponding state, then
		\[
		\inf_{(M,\tau, \alpha_1, \alpha_2)}  \frac{1}{n} \sum_{i,j=1}^n \frac{1}{2} \left \| \alpha_1(u_{ij}) - \alpha_2(u_{ij}) \right \|_{L^1(M, \tau)} = \varphi(C_1).
		\]
		Since $\varphi \mapsto \varphi(C_1)$ is weak-$*$ continuous and $\mathcal{T}_{\varphi_1,\varphi_2}(A * A)$ is compact, the infimum is achieved.
	\end{proof}
	
	We now record the properties of $\ldistance$ analogous to those of $\freedistance$.
	
	\begin{proposition} \label{prop: L distance properties}
		Let $\varphi_1$, $\varphi_2$, $\psi_1$, $\psi_2 \in \mathcal{T}(C(S_n^+))$.
		\begin{enumerate}
			\item $\ldistance$ is a metric on $\cT(C(S_n^+))$.
			\item $\ldistance$ is weak-$*$ lower semi-continuous on $\mathcal{T}(C(S_n^+)) \times \mathcal{T}(C(S_n^+))$ and convergence in $\ldistance$ implies weak-$*$ convergence.
			\item $\ldistance(\varphi_1 * \psi_1, \varphi_2 * \psi_2) \leq \ldistance(\varphi_1,\varphi_2) + \ldistance(\psi_1, \psi_2)$.
		\end{enumerate}
	\end{proposition}
	
	\begin{proof}
		(1) $\ldistance(\varphi_1, \varphi_2) \geq 0$ is straightforward. For $\varphi_1 = \varphi_2$, choose the tracial coupling $(M, \tau, \alpha_1, \alpha_2)$ such that $\tau \circ \alpha_1 = \varphi_1 = \varphi_2 = \tau \circ \alpha_2$. So, for each $i,j$, $\tau(\alpha_1(u_{ij}) - \alpha_2(u_ij)) = \tau(\alpha_1 - \alpha_2 (u_{ij}^* u_{ij})) = \tau (\alpha_1(u_{ij}) - \alpha_2(u_{ij}))^* (\alpha_1(u_{ij}) - \alpha_2(u_{ij})) $. Since $\tau$ is faithful, we get $\alpha_1(u_{ij}) = \alpha_2(u_{ij})$ for all $i,j$, so $\ldistance (\varphi_1, \varphi_2) = 0$. By Lemma \ref{lem:inf-l1distance}, the infimum is achieved $\ldistance(\varphi_1, \varphi_2) = 0 \implies $ there is a coupling $(M, \tau, \alpha_1, \alpha_2)$ such that $\| \alpha_1(u_ij)-\alpha_2(u_ij)\| = 0 \implies \alpha_1(u_{ij}) = \alpha_2(u_{ij})$ for all $i,j$ and therefore $\alpha_1 = \alpha_2$. Then $\varphi_1=\tau \circ \alpha_1 = \tau \circ \alpha_2 = \varphi_2$. Symmetry is similar to \cref{prop: freedistance is a metric}. The triangle inequality follows from the triangle inequality for the norm of $L^1(M,\tau)$ and constructing the amalgamated free product as in the case of $\freedistance$.
		
		(2) The proof is completely analogous to Lemma \ref{lem: weak* semicontinuity 1}, so we leave the details to the reader.
		
		(3) As in Proposition \ref{prop: convolution}, consider tracial couplings $(M,\tau,\alpha_1,\alpha_2)$ of $\varphi_1$ and $\varphi_2$ and $(N,\sigma,\beta_1,\beta_2)$ of $\psi_1$ and $\psi_2$.  Let $\gamma_j = (\alpha_j \otimes \beta_j) \circ \Delta: C(S_n^+) \to M \overline{\otimes} N$, so that $(M \overline{\otimes} N, \tau \otimes \sigma, \gamma_1,\gamma_2)$ is a tracial coupling of $\varphi_1 * \psi_1$ and $\varphi_2 * \psi_2$.  Note that $\norm{x \otimes y}_{L^1(M \overline{\otimes} N, \tau \otimes \sigma)} = \norm{x}_{L^1(M,\tau)} \norm{y}_{L^1(N,\sigma)}$.  Thus,
		\begin{align*}
			\lVert \gamma_1(u_{ij}) & - \gamma_2(u_{ij}) \rVert_{L^1(M \overline{\otimes} N,\tau \otimes \sigma)} \\
			&=  \norm{(\alpha_1 \otimes \beta_1)\Delta(u_{ij}) - (\alpha_2 \otimes \beta_2)\Delta(u_{ij})}_{L^1(M \overline{\otimes} N,\tau \otimes \sigma)} \\
			& = \norm*{ \sum_{k=1}^n \alpha_1(u_{ik}) \otimes \beta_1(u_{kj}) - \alpha_2(u_{ik}) \otimes \beta_2(u_{kj}) }_{L^1(M \overline{\otimes} N,\tau \otimes \sigma)}
			\\
			& = \norm*{ \sum_{k=1}^n (\alpha_1(u_{ik}) - \alpha_2(u_{ik})) \otimes \beta_1(u_{kj}) + \alpha_2(u_{ik}) \otimes (\beta_1(u_{kj}) - \beta_2(u_{kj})) }_{L^1(M \overline{\otimes} N,\tau \otimes \sigma)}
			\\
			&\leq \sum_{k=1}^n \norm{ \alpha_1(u_{ik}) - \alpha_2(u_{ik})}_{L^1(M,\tau)} \norm{\beta_1(u_{kj})}_{L^1(N,\sigma)} + \norm{\alpha_2(u_{ik})}_{L^1(M,\tau)} \norm{ \beta_1(u_{kj}) - \beta_2(u_{kj})}_{L^1(N,\sigma)} \\
			&= \sum_{k=1}^n \norm{ \alpha_1(u_{ik}) - \alpha_2(u_{ik})}_{L^1(M,\tau)} \sigma(\beta_1(u_{kj})) + \tau(\alpha_2(u_{ik})) \norm{ \beta_1(u_{kj}) - \beta_2(u_{kj})}_{L^1(N,\sigma)}.
		\end{align*}
		Therefore,
		\begin{align*}
			& \quad \frac{1}{2n} \sum_{i,j=1}^n \lVert \gamma_1(u_{ij}) - \gamma_2(u_{ij}) \rVert_{L^1(M \overline{\otimes} N,\tau \otimes \sigma)} \\
			&\leq \frac{1}{2n} \sum_{i,j,k=1}^n \norm{ \alpha_1(u_{ik}) - \alpha_2(u_{ik})}_{L^1(M,\tau)} \sigma(\beta_1(u_{kj})) + \frac{1}{2n} \sum_{i,j,k=1}^n \tau(\alpha_2(u_{ik})) \norm{ \beta_1(u_{kj}) - \beta_2(u_{kj})}_{L^1(N,\sigma)} \\
			&= \frac{1}{2n} \sum_{i,k=1}^n \norm{ \alpha_1(u_{ik}) - \alpha_2(u_{ik})}_{L^1(M,\tau)} \left[ \sum_{j=1}^n \sigma(\beta_1(u_{kj})) \right] + \frac{1}{2n} \sum_{j,k=1}^n \left[ \sum_{i=1}^n \tau(\alpha_2(u_{ik})) \right] \norm{ \beta_1(u_{kj}) - \beta_2(u_{kj})}_{L^1(N,\sigma)} \\
			&= \frac{1}{2n} \sum_{i,k=1}^n \norm{ \alpha_1(u_{ik}) - \alpha_2(u_{ik})}_{L^1(M,\tau)}  + \frac{1}{2n} \sum_{j,k=1}^n \norm{ \beta_1(u_{kj}) - \beta_2(u_{kj})}_{L^1(N,\sigma)}.
		\end{align*}
		Since the couplings were arbitrary, we obtain $\ldistance(\varphi_1 * \psi_1, \varphi_2 * \psi_2) \leq \ldistance(\varphi_1,\varphi_2) + \ldistance(\psi_1,\psi_2)$ as desired.
	\end{proof}

	\section{Distance via tensor couplings} \label{sec: tensor distance}
	
	\subsection{Tracial tensor couplings}
	
	\begin{definition}
		Let $A$ be a $\mathrm{C}^*$-algebra, and let $\varphi_1$, $\varphi_2 \in \mathcal{T}(A)$.  A \emph{tracial tensor coupling} of $(\varphi_1,\varphi_2)$ is a trace $\varphi \in \mathcal{T}(A \otimes A)$ such that $\varphi(a \otimes 1) = \varphi_1(a)$ and $\varphi(1 \otimes a) = \varphi_2(a)$ for $a \in A$.  We denote the set of tracial tensor couplings by $\Pi_{\tr,\otimes}(\varphi_1,\varphi_2)$.
	\end{definition}
	
	\begin{definition}
		A \emph{cost operator} is an element $C \in (A \otimes A)_+$.  We say that $C$ is \emph{symmetric} if $C_{2,1} = C$.  We say that $C$ satisfies the \emph{triangle inequality} if $\iota_{1,3}(C) \leq \iota_{1,2}(C) + \iota_{2,3}(C)$, where we employ tensor-leg notation.
	\end{definition}
	
	\begin{definition}
		For a cost operator $C$ and $\varphi_1$, $\varphi_2 \in \mathcal{T}(A)$, define
		\[
		W_C(\varphi_1,\varphi_2) = \inf_{\varphi \in \Pi_{\tr,\otimes}(\varphi_1,\varphi_2)} \varphi(C).
		\]
	\end{definition}
	
	\begin{remark}
		The infimum is always achieved because $\Pi_{\tr,\otimes}$ is a weak-$*$ compact subset of $\mathcal{T}(A)$ and $\varphi \mapsto \varphi(C)$ is weak-$*$ continuous.
	\end{remark}
	
	\begin{proposition} \label{prop: general tensor coupling properties}
		Let $C \in (A \otimes A)_+$ be a cost operator and $\varphi_1$, $\varphi_2$, $\varphi_3 \in \mathcal{T}(A)$.
		\begin{enumerate}
			\item $W_C(\varphi_1,\varphi_2) \geq 0$.
			\item If $C$ is symmetric, then $W_C(\varphi_1,\varphi_2) = W_C(\varphi_2,\varphi_1)$.
			\item If $C$ satisfies the triangle inequality, then $W_C(\varphi_1,\varphi_3) \leq W_C(\varphi_1,\varphi_2) + W_C(\varphi_2,\varphi_3)$.
		\end{enumerate}
	\end{proposition}
	
	The main challenge in Proposition \ref{prop: general tensor coupling properties} is to prove the triangle inequality.  In order to do this, we need to consider $\tau_1 \in \Pi_{\tr,\otimes}(\varphi_1,\varphi_2)$ and $\tau_2 \in \Pi_{\tr,\otimes}(\varphi_2,\varphi_3)$ and amalgamate them to obtain an element of $\Pi_{\tr,\otimes}(\varphi_1,\varphi_3)$.
	
	The key property that makes this proposition work is that the trace space is a \emph{Choquet simplex}.  Recall that for a weak-$*$ closed convex subset $X$ in the dual of some Banach space, we denote by $\partial_e X$ the set of extreme points.  Choquet's theorem (see e.g.\ \cite[\S I.4]{Alfsen1971}) says for each $x \in X$, there exists a probability measure $\mu$ on $\partial_e X$ such that $x = \int_{\partial_e X} y\,d\mu(y)$; in this case, we say that $x$ is the \emph{barycenter} of $\mu$.  If the measure $\mu$ is always \emph{unique}, then $X$ is called a Choquet simplex (see e.g.\ \cite[\S II.3]{Alfsen1971}).  For a $\mathrm{C}^*$-algebra $A$, both $\mathcal{S}(A)$ and $\mathcal{T}(A)$ are weak-$*$ closed convex subsets of the dual of $A$.  However, $\mathcal{S}(A)$ is never a Choquet simplex unless $A$ is commutative.  On the other hand, $\mathcal{T}(A)$ is always a Choquet simplex.  A simple proof of this fact is given in \cite{BlackadarRordam2024}.
	
	\begin{lemma} \label{lem: trace decomposition}
		Let $A$ be a unital $\mathrm{C}^*$-algebra.  For every $\varphi \in \mathcal{T}(A)$, there exists a unique $\mu \in \mathcal{P}(\partial_e \mathcal{T}(A))$ such that
		\[
		\varphi(a) = \int_{\partial_e \mathcal{T}(A)} \tau(a) \,d\mu(\tau).
		\]
	\end{lemma}
	
	In order to study traces on the tensor product, we also need to use the following fact.
	
	\begin{lemma} \label{lem: tensor product traces}
		Let $A$ and $B$ be $\mathrm{C}^*$-algebras.  Let $\tau \in \mathcal{T}(A \otimes B)$.  Then $\tau$ is extreme if and only if $\tau = \tau_1 \otimes \tau_2$ where $\tau_1 \in \partial_e \mathcal{T}(A)$ and $\tau_2 \in \partial_e \mathcal{T}(B)$.  Moreover, this relationship defines a weak-$*$ homeomorphism $\partial_e \mathcal{T}(A \otimes B) \cong \partial_e \mathcal{T}(A) \times \partial_e \mathcal{T}(B)$.
	\end{lemma}

	\begin{lemma} \label{lem: tracial coupling glued}
		Let $A$ be a unital $\mathrm{C}^*$-algebra.  Let $\varphi_1$, $\varphi_2$, $\varphi_3 \in \mathcal{T}(A)$.  Fix couplings $\rho_1 \in \Pi_{\tr}(\varphi_1,\varphi_2)$ and $\rho_2 \in \Pi_{\tr}(\varphi_2,\varphi_3)$.  Then there exists a unique $\sigma \in \mathcal{T}(A \otimes A \otimes A)$ such that
		\[
		\rho \circ \iota_{1,2} = \rho_1, \qquad \rho \circ \iota_{2,3} = \rho_2,
		\]
		where $\iota_{i,j}$ is the natural tensor leg inclusion.
	\end{lemma}
	
	\begin{proof}
		Recall that $\partial_e \mathcal{T}(A \otimes A) \cong \partial_e \mathcal{T}(A) \otimes \partial_e \mathcal{T}(A)$, and let $\pi_1$, $\pi_2: \partial_e \mathcal{T}(A \otimes A) \to \partial_e \mathcal{T}(A)$ be the two canonical coordinate projections.
		
		For $j = 1$, $2$, $3$ let $\mu_j \in P(\partial_e \mathcal{T}(A))$ be the unique measure such that $\varphi_j = \int_{\partial_e \mathcal{T}(A)} \tau \,d\mu_j(\tau)$.  Similarly, for $j = 1$, $2$, let $\nu_j \in P(\partial_e \mathcal{T}(A \otimes A))$ be the unique measure such that $\rho_j$ is the barycenter of $\nu_j$.  Note that $(\pi_2)_* \nu_1$ gives a measure on $\partial_e \mathcal{T}(A)$ whose barycenter is $\varphi_2$, and therefore since $\mathcal{T}(A)$ is a Choquet simplex, we have $(\pi_2)_* \nu_1 = \mu_2$.  Similarly, $(\pi_1)_* \nu_2 = \mu_2$.  Thus, $\nu_1$ and $\nu_2$ yield are two probability measures on $\partial_e \mathcal{T}(A) \times \partial_e \mathcal{T}(A)$ such that the second marginal of $\nu_1$ agrees with the first marginal of $\nu_2$.  By standard results on conditional independence in probability theory, there exists a probability measure $\lambda$ on $\partial_e \mathcal{T}(A) \times \partial_e \mathcal{T}(A) \times \partial_e \mathcal{T}(A) \cong \partial_e \mathcal{T}(A \otimes A \otimes A)$ whose projection on the first two coordinates is $\nu_1$ and whose projection on the second two coordinates is $\nu_2$.  Then let
		\[
		\rho = \int_{\partial_e \mathcal{T}(A)^{\times 3}} \tau_1 \otimes \tau_2 \otimes \tau_3 \,d\lambda(\tau_1,\tau_2,\tau_3).
		\]
		Then
		\[
		\rho \circ \iota_{1,2} = \int_{\partial_e \mathcal{T}(A)^{\times 3}} \tau_1 \otimes \tau_2 \,d\lambda(\tau_1,\tau_2,\tau_3) = \int_{\partial_e \mathcal{T}(A)^{\times 2}} \tau_1 \otimes \tau_2 \,d\nu_1(\tau_1,\tau_2) = \rho_1,
		\]
		and similarly $\rho \circ \iota_{2,3} = \rho_2$.
	\end{proof}
	
	\begin{proof}[Proof of Proposition \ref{prop: general tensor coupling properties}]
		(1) follows because for any $\varphi \in \Pi_{\tr,\otimes}(\varphi_1,\varphi_2)$, we have $\varphi(C) \geq 0$ by positivity of $C$.
		
		For (2), note that if $\varphi \in \Pi_{\tr,\otimes}$, then so is $\varphi$ composed with the tensor flip.
		
		It remains to prove (3) that $W_C$ satisfies the triangle inequality provided that $C_{1,3} \leq C_{1,2} + C_{2,3}$.  Let $\varphi_1$, $\varphi_2$, $\varphi_3 \in \mathcal{T}(A)$.  Fix $\psi_1 \in \Pi_{\tr,\otimes}(\varphi_1,\varphi_2)$ and $\psi_2 \in \Pi_{\tr,\otimes}(\varphi_2,\varphi_3)$.  By Lemma \ref{lem: tracial coupling glued}, there exists some $\rho \in \mathcal{T}(A \otimes A \otimes A)$ such that
		\[
		\rho \circ \iota_{1,2} = \psi_1, \qquad \rho \circ \iota_{2,3} = \psi_2.
		\]
		Note that $\psi := \rho \circ \iota_{1,3} \in \Pi_{\tr,\otimes}(\varphi_1,\varphi_3)$.  Therefore,
		\begin{align*}
			\rho(\iota_{1,3}(C)) &\leq \rho(\iota_{1,2}(C)) + \rho(\iota_{2,3}(C)) \\
			\psi(C) &\leq \psi_1(C) + \psi_2(C).
		\end{align*}
		Hence, taking the infimum over $\psi_1 \in \Pi_{\tr,\otimes}(\varphi_1,\varphi_2)$ and $\psi_2 \in \Pi_{\tr,\otimes}(\varphi_2,\varphi_3)$, we have
		\[
		W_C(\varphi_1,\varphi_3) \leq W_C(\varphi_1,\varphi_2) + W_C(\varphi_2,\varphi_3),
		\]
		as desired.
	\end{proof}
	
	\begin{remark}
		In fact, both properties (1) and (2) also work if we use general states rather than traces, and consider the set $\Pi_\otimes(\varphi_1,\varphi_2)$ of states $\varphi$ on $A \otimes A$ with $\varphi(a \otimes 1) = \varphi_1(a)$ and $\varphi(1 \otimes a) = \varphi_2(a)$.  Moreover, the triangle inequality (3) will hold whenever the middle state $\varphi_2$ admits a unique decomposition in terms of extreme points, e.g.\ if $\varphi_2$ is itself extreme.  Compare \cite[Theorem 1]{BunthPitrikTitkosVirosztek2024} where certain quantum Wasserstein distances are shown to satisfy the triangle inequality when the middle state is pure.
	\end{remark}
	
	We finally remark that by decomposing a trace in terms of extreme traces, the tracial couplings can be reduced in some sense to classical couplings.  This lemma follows from the same reasoning used in the proof of Lemma \ref{lem: tracial coupling glued}.
	
	\begin{lemma} \label{lem: classical decomposition of tensor cost}
		Let $C$ be a cost operator, and for $\tau_1$, $\tau_2 \in \partial_e \mathcal{T}(A)$, define
		\[
		c(\tau_1,\tau_2) := (\tau_1 \otimes \tau_2)(C).
		\]
		Let $\varphi_1$, $\varphi_2 \in \mathcal{T}(A)$.  Let $\mu_1$ and $\mu_2$ be the unique probability measures on $\partial_e \mathcal{T}(A)$ with barycenters $\varphi_1$ and $\varphi_2$.  Then any $\varphi \in \Pi_{\tr,\otimes}(\varphi_1,\varphi_2)$ is the barycenter of a unique probability measure $\mu$ on $\partial_e \mathcal{T}(A) \times \partial_e \mathcal{T}(A)$ which has marginals $\mu_1$ and $\mu_2$, and we have
		\[
		\varphi(C) = \int_{\partial_e \mathcal{T}(A) \times \partial_e \mathcal{T}(A)} c(\tau_1,\tau_2) \,d\mu(\tau_1,\tau_2).
		\]
	\end{lemma}

	\subsection{Tensor Hamming distance for $C(S_n^+)$} \label{subsec: tensor Hamming}
	
	\begin{definition}
		The Hamming cost operator is the element $C_H \in C(S_n^+) \otimes C(S_n^+)$ given by
		\[
		C_H = 1 - \frac{1}{n} \sum_{i,j=1}^n u_{ij} \otimes u_{ij}.
		\]
	\end{definition}
	
	\begin{remark}
		The Hamming cost for quantum permutations can be expressed in another way as follows.  Let $U$ be the canonical $n \times n$ unitary matrix in $M_n(\C) \otimes C(S_n^+)$.  Then
		\[
		\frac{1}{n} \sum_{i,j=1}^n \rho(u_{ij} \otimes u_{ij}) = [\tr_n \otimes \rho](\iota_{1,2}(U^*) \iota_{1,3}(U)),
		\]
		where the element on the right-hand side is in $M_n(\mathbb{C}) \otimes C(S_n^+) \otimes C(S_n^+)$.  Therefore,
		\[
		C_H = 1 - [\tr_n \otimes \rho](\iota_{1,2}(U^*) \iota_{1,3}(U)).
		\]
		Since $U$ is unitary, it is clear that $[\tr_n \otimes \rho](\iota_{1,2}(U^*) \iota_{1,3}(U))$ has norm at most $1$.  We also have that $C_H$ is self-adjoint since the $u_{ij}$'s are projections.  Therefore, $C_H \geq 0$, so it is a valid cost operator.
	\end{remark}
	
	\begin{lemma} \label{lem: tensor metric properties}
		The Hamming cost $C_H$ is symmetric and satisfies the triangle inequality $\iota_{1,3}(C_H) \leq \iota_{1,2}(C_H) + \iota_{2,3}(C_H)$.  Consequently, $\tensordistance$ is nonnegative and symmetric, and it satisfies the triangle inequality.
	\end{lemma}
	
	\begin{proof}
		We already verified $C_H \geq 0$, and it is immediate that $C_H$ is symmetric under the tensor flip.  It remains to show that $C_H$ satisfies the triangle inequality, that is,
		\[
		1 - \frac{1}{n} \sum_{i,j=1}^n u_{ij} \otimes 1 \otimes u_{ij} \leq \left( 1 - \frac{1}{n} \sum_{i,j=1}^n u_{ij} \otimes u_{ij} \otimes 1 \right) + \left( 1 - \frac{1}{n} \sum_{i,j=1}^n 1 \otimes u_{ij} \otimes u_{ij} \right).
		\]
		It suffices to prove for each $i$ that
		\[
		1 - \sum_{j=1}^n u_{ij} \otimes 1 \otimes u_{ij} \leq \left( 1 - \sum_{j=1}^n u_{ij} \otimes u_{ij} \otimes 1 \right) + \left( 1 - \sum_{j=1}^n 1 \otimes u_{ij} \otimes u_{ij} \right),
		\]
		since of course this operator inequality can be summed over $i$ then divided by $n$.  Therefore, fix $i$, and write
		\begin{align*}
			p_{1,2} &= \sum_{j=1}^n u_{ij} \otimes u_{ij} \otimes 1 \\
			p_{2,3} &= \sum_{j=1}^n 1 \otimes u_{ij} \otimes u_{ij}  \\
			p_{1,3} &= \sum_{j=1}^n u_{ij} \otimes 1 \otimes u_{ij}.
		\end{align*}
		Then we want to prove that $1 - p_{1,3} \leq (1 - p_{1,2}) + (1 - p_{2,3})$.  Recall that $(u_{ij})_{j=1}^n$ is a family of projections that sum up to $1$, and in particular, they commute.  Similarly, $(u_{ji})_{j=1}^n$ is a family of projections that add up to $1$.  Therefore, for a fixed value of $i$, all the operators in the equations above commute with each other.  Note that $p_{1,2}$ is a projection since $u_{ij} \otimes u_{ij}$ for $j = 1$, \dots, $n$ are disjoint projections.  Similarly, the other four operators are projections.  Since we have commuting projections,
		\[
		(1 - p_{1,2}) + (1 - p_{2,3}) \geq (1 - p_{1,2}) \vee (1 - p_{2,3}) = 1 - p_{1,2} p_{2,3}.
		\]
		Using again the fact $\{u_{ij}\}_{j=1}^n$ are orthogonal
		\begin{align*}
			p_{1,2} p_{2,3} &= \left( \sum_{j=1}^n u_{ij} \otimes u_{ij} \otimes 1 \right) \left( \sum_{j'=1}^n 1 \otimes u_{ij'} \otimes u_{ij'} \right) \\
			&= \sum_{j=1}^n u_{ij} \otimes u_{ij} \otimes u_{ij} \\
			&\leq \sum_{j=1}^n u_{ij} \otimes 1 \otimes u_{ij} \\
			&= p_{1,3}.
		\end{align*}
		Thus, $1 - p_{1,3} \leq 1 - p_{1,2} p_{2,3} \leq (1 - p_{1,2}) + (1 - p_{2,3})$, as desired.
	\end{proof}
	
	\subsection{Bounds in terms of self-cost and equality cases} \label{subsec: equality cases}
	
	Lemma \ref{lem: classical decomposition of tensor cost} allows us to rephrase the tensor Hamming distance for the quantum permutation group in terms of a classical optimization problem.
	
	\begin{lemma} \label{lem: classical decomposition of tensor cost 2}
		Let $\varphi_1$, $\varphi_2 \in \mathcal{T}(C(S_n^+))$.  Let $\mu_j$ be the unique probability measure on $\partial_e \mathcal{T}(C(S_n^+))$ with barycenter $\varphi_j$.  Let
		\[
		c_H: \partial_e \mathcal{T}(C(S_n^+)) \times \partial_e \mathcal{T}(C(S_n^+)) \to [0,\infty)
		\]
		be given by
		\[
		c_H(\tau_1,\tau_2) = (\tau_1 \otimes \tau_2)(C_H) = 1 - \frac{1}{n} \sum_{i,j=1}^n \tau_1(u_{ij}) \tau_2(u_{ij}).
		\]
		Then
		\[
		\tensordistance(\varphi_1,\varphi_2) = \inf_\mu \int_{\partial_e \mathcal{T}(C(S_n^+)) \times \partial_e \mathcal{T}(C(S_n^+))} c_H(\tau_1,\tau_2)\,d\mu(\tau_1,\tau_2),
		\]
		where $\mu$ ranges over probability measures on $\partial_e \mathcal{T}(C(S_n^+)) \times \partial_e \mathcal{T}(C(S_n^+))$ with marginals $\mu_1$ and $\mu_2$ respectively.
	\end{lemma}
	
	\begin{proof}
		The evaluation of $(\tau_1 \otimes \tau_2)(C_H) = 1 - \frac{1}{n} \sum_{i,j=1}^n \tau_1(u_{ij}) \tau_2(u_{ij})$ follows by direct computation.  By Lemma \ref{lem: classical decomposition of tensor cost}, the tensor couplings $\varphi$ of $\varphi_1$ and $\varphi_2$ are in bijection with probability measures $\mu$ on $\partial_e \mathcal{T}(C(S_n^+)) \times \partial_e \mathcal{T}(C(S_n^+))$, and the cost of each $\mu$ is given by the expression asserted in the lemma.
	\end{proof}
	
	In particular, this allows us to evaluate the self-distance $\tensordistance(\varphi,\varphi)$ as follows.  Note the parallel with the self-distance estimate for a certain quantum Wasserstein distance \cite[Theorem 1]{DePalmaTrevisan2021}.
	
	\begin{lemma} \label{lem: self distance estimate}
		Let $\varphi_0 \in \mathcal{T}(C(S_n^+))$ and let $\mu_0$ be the corresponding probability measure on $\partial_e \mathcal{T}(C(S_n^+))$.  Then
		\begin{equation} \label{eq: self-distance evaluation}
			\tensordistance(\varphi_0,\varphi_0) = 1 - \frac{1}{n} \sum_{i,j=1}^n \int_{\partial_e \mathcal{T}(C(S_n^+))} \tau(u_{ij})^2 \,d\mu_0(\tau).
		\end{equation}
		Now let $\varphi_1$, $\varphi_2 \in \mathcal{T}(C(S_n^+))$ and let $\mu_1$ and $\mu_2$ be the corresponding probability measures on $\partial_e \mathcal{T}(C(S_n^+))$. Then
		\begin{multline} \label{eq: self-distance estimate}
			\tensordistance(\varphi_1,\varphi_2) - \frac{1}{2} \tensordistance(\varphi_1,\varphi_1) - \frac{1}{2} \tensordistance(\varphi_2,\varphi_2) \\
			= \inf_{\mu}  \int_{\partial_e \mathcal{T}(C(S_n^+)) \times \partial_e \mathcal{T}(C(S_n^+))} \frac{1}{2n} \sum_{i,j=1}^n (\tau_1(u_{ij}) - \tau_2(u_{ij}))^2\,d\mu(\tau_1,\tau_2).
		\end{multline}
	\end{lemma}
	
	\begin{proof}
		For ease of notation, write $\Omega = \partial_e \mathcal{T}(C(S_n^+))$.  Consider $\varphi_1$, $\varphi_2 \in \mathcal{T}(C(S_n^+))$ and corresponding measures $\mu_1$ and $\mu_2$, and let $\mu$ be a measure on $\partial_e \mathcal{T}(C(S_n^+)) \times \partial_e \mathcal{T}(C(S_n^+))$ with marginals $\mu_1$ and $\mu_2$.  Then we have by the arithmetic-geometric mean inequality,
		\begin{align*}
			\int_{\Omega \times \Omega} & \left[ 1 - \frac{1}{n} \sum_{i,j=1}^n \tau_1(u_{ij}) \tau_2(u_{ij})\right] \,d\mu(\tau_1,\tau_2) \\
			&= 1 - \frac{1}{n} \sum_{i,j=1}^n \int_{\Omega \times \Omega} \tau_1(u_{ij}) \tau_2(u_{ij})\,d\mu(\tau_1,\tau_2) \\
			&\geq 1 - \frac{1}{n} \sum_{i,j=1}^n \int_{\Omega \times \Omega} \left[ \frac{1}{2} \tau_1(u_{ij})^2 + \frac{1}{2} \tau_2(u_{ij})^2 \right] \,d\mu(\tau_1,\tau_2) \\
			&= \frac{1}{2} \int_{\Omega} \left[ 1 - \frac{1}{n} \sum_{i,j=1}^n \tau_1(u_{ij})^2 \right]\,d\mu_1(\tau_1) + \frac{1}{2} \int_{\Omega} \left[ 1 - \frac{1}{n} \sum_{i,j=1}^n \tau_2(u_{ij})^2 \right]\,d\mu_2(\tau_2).
		\end{align*}
		In particular, in the case when $\mu_1 = \mu_2 = \mu_0$, we obtain
		\[
		\int_{\Omega \times \Omega} \left[ 1 - \frac{1}{n} \sum_{i,j=1}^n \tau_1(u_{ij}) \tau_2(u_{ij})\right] \,d\mu(\tau_1,\tau_2) \geq \int_{\Omega} \left[ 1 - \frac{1}{n} \sum_{i,j=1}^n \tau(u_{ij})^2 \right]\,d\mu_0(\tau),
		\]
		and thus, since $\mu$ was arbitrary,
		\[
		\tensordistance(\varphi_0,\varphi_0) \geq \int_{\partial_e \mathcal{T}(C(S_n^+))} \left[ 1 - \frac{1}{n} \sum_{i,j=1}^n \tau(u_{ij})^2 \right]\,d\mu_0(\tau).
		\]
		On the other hand, equality is achieved when we take $\mu$ to be a copy of $\mu_0$ supported on the diagonal of $\Omega \times \Omega$.  Therefore, we have proved \eqref{eq: self-distance evaluation}.
		
		We then observe that
		\begin{align*}
			\int_{\Omega \times \Omega} & \left[ 1 - \frac{1}{n} \sum_{i,j=1}^n \tau_1(u_{ij}) \tau_2(u_{ij})\right] \,d\mu(\tau_1,\tau_2) - \frac{1}{2} \tensordistance(\varphi_1,\varphi_1) - \frac{1}{2} \tensordistance(\varphi_2,\varphi_2) \\
			&= \int_{\Omega \times \Omega} \left[ 1 - \frac{1}{n} \sum_{i,j=1}^n \tau_1(u_{ij}) \tau_2(u_{ij})\right] \,d\mu(\tau_1,\tau_2) \\
			& \quad - \frac{1}{2} \int_{\Omega} \left[ 1 - \frac{1}{2} \sum_{i,j=1}^n \tau_1(u_{ij})^2 \right] \,d\mu_1(\tau_1) - \frac{1}{2} \int_{\Omega} \left[1 - \frac{1}{n} \sum_{i,j=1}^n \tau_2(u_{ij})^2\right] \,d\mu_2(\tau_2) \\
			&= \int_{\Omega \times \Omega} \frac{1}{2n} \sum_{i,j=1}^n (\tau_1(u_{ij}) - \tau_2(u_{ij}))^2\,d\mu(\tau_1,\tau_2).
		\end{align*}
		By taking the infimum over $\mu$ with marginals $\mu_1$ and $\mu_2$ on both sides of the equation, we obtain \eqref{eq: self-distance estimate}.
	\end{proof}
	
	As a consequence of this formula, we next observe that the self-distance vanishes if and only if $\varphi$ arises from a classical probability measure on $S_n$.  Thus, ``most'' tracial states on $C(S_n^+)$ have nonzero self-distance.
	
	\begin{lemma}
		Let $\varphi$ be a tracial state on $C(S_n^+)$.  Then $\tensordistance(\varphi,\varphi) = 0$ if and only if $\varphi$ has the form $\varphi(a) = \int_{S_n} \theta(a)\,d\nu$ for some $\nu \in P(S_n)$, where $\theta: C(S_n^+) \to C(S_n)$ is the canonical quotient map.
	\end{lemma}
	
	\begin{proof}
		Suppose that $\tensordistance(\varphi,\varphi) = 0$.  Let $\mu$ be the measure on $\partial_e \mathcal{T}(C(S_n^+))$ corresponding to $\varphi$.  Then
		\begin{align*}
			\int_{\partial_e \mathcal{T}(C(S_n^+))} \frac{1}{n} \sum_{i,j=1}^n [\tau(u_{ij}) - \tau(u_{ij})^2]\,d\mu(\tau) &= \int_{\partial_e \mathcal{T}(C(S_n^+))} \left[ \frac{1}{n} \sum_{i=1}^n \tau(1) - \frac{1}{n} \sum_{i,j=1}^n \tau(u_{i,j})^2] \right] \,d\mu(\tau) \\
			&= 1 - \int_{\partial_e \mathcal{T}(C(S_n^+))} \left[ \frac{1}{n} \sum_{i,j=1}^n \tau(u_{ij})^2] \right] \,d\mu(\tau) \\
			&= \tensordistance(\varphi,\varphi).
		\end{align*}
		Since $0 \leq \tau(u_{ij}) - \tau(u_{ij})^2 \leq 1$, we conclude that for $\mu$-almost every $\tau$, for all $i,j$, we have $\tau(u_{ij}) = \tau(u_{ij})^2$, and so $\tau(u_{ij}) \in \{0,1\}$.  For any $\tau$ that satisfies this condition, the projection $\pi_{\tau}(u_{ij})$ is either $0$ or $1$ by faithfulness of the trace on $\pi_{\tau}(A)''$.  Since $\pi_{\tau}(A)''$ is generated by $\{\pi_{\tau}(u_{ij})\}_{i,j=1}^n$, this forces\ $\pi_{\tau}(A)''$ to be one-dimensional, and so $\pi_{\tau}(a) = \tau(a) 1$ for all $a$.  Hence, $\tau$ is $*$-homomorphism $C(S_n^+) \to \C$ and so it vanishes on $\ker(\theta)$.  Since this holds for $\mu$-almost every $\tau$, we see that $\varphi$ vanishes on $\ker(\theta)$.  Hence, there exists a tracial state $\overline{\varphi}$ on $C(S_n)$ such that $\varphi = \overline{\varphi} \circ \theta$; indeed, one can verify readily that $\overline{\varphi}$ is well-defined as a linear map since $\varphi$ vanishes on the kernel and then check that $\overline{\varphi}$ is unital, positive, and tracial.  But every tracial state on $C(S_n)$ is given by a probability measure on $S_n$.
		
		For the converse direction, we defer to Proposition \ref{prop: recovery of classical distance} below which shows that for tracial states that arise from probability measures on $S_n$, our $\tensordistance$ agrees with the classical $L^1$ Wasserstein distance, so in particular the self-distance is zero.  Alternatively, one can argue that the Dirac delta masses on $C(S_n)$ produce extreme states on $C(S_n^+)$, so the decomposition of $\nu$ in terms of such Dirac masses gives a direct description of boundary measure $\mu$ to use in \eqref{eq: self-distance evaluation}.
	\end{proof}

	The preceding lemma allows us to express $\tensordistance$ in terms of measures on the Birkhoff polytope
	\[
	B_n = \{A \in M_n(\mathbb{R}): A_{ij} \in [0,1], \sum_i A_{ij} = 1, \sum_j A_{ij} = 1 \}
	\]
	of bistochastic matrices. We remark that if $\varphi$ is a state on $C(S_n^+)$, then the matrix
	\[
	\omega(\varphi) := [\varphi(u_{ij})]_{i,j=1}^n \text{ is in } B_n
	\]
	because $\sum_i u_{ij} = 1$ and $\sum_j u_{ij} = 1$.
	
	\begin{lemma}
		Let $\varphi_1$, $\varphi_2 \in \mathcal{T}(C(S_n^+))$, and let $\mu_j$ be the probability measure on $\partial_e \mathcal{T}(C(S_n^+))$ with barycenter $\varphi_j$.  Let $\omega_* \mu_j \in \mathcal{P}(B_n)$ be the pushforward and let $W_2$ be the Wasserstein distance on $\mathcal{P}(B_n)$ with respect to the normalized Hilbert-Schmidt metric $\norm{A - B}_2 = \tr_n[(A - B)^2]^{1/2}$.  Then we have
		\[
		\tensordistance(\varphi_1,\varphi_2) - \frac{1}{2} \tensordistance(\varphi_1,\varphi_1) - \frac{1}{2} \tensordistance(\varphi_2,\varphi_2) = \frac{1}{2} W_2(\omega_* \mu_1,\omega_* \mu_2)^2.
		\]
		In particular, the left-hand side vanishes if and only if $\omega_* \mu_1 = \omega_* \mu_2$.
	\end{lemma}
	
	\begin{proof}
		First, let $\varphi$ be a trace that achieves the optimum in the definition of $\tensordistance(\varphi_1,\varphi_2)$, and let $\mu$ be the corresponding probability measure on $\partial_e \mathcal{T}(C(S_n^+)) \times \partial_e \mathcal{T}(C(S_n^+))$.  Then $\nu = (\omega \times \omega)_* \mu$ defines a classical coupling of $\nu_1 = \omega_* \mu_1$ and $\nu_2 = \omega_* \nu_2$.  It follows from \eqref{eq: self-distance estimate} and the definition of $\nu_1$, $\nu_2$, that
		\begin{align*}
			\tensordistance(\varphi_1,\varphi_2) & - \frac{1}{2} \tensordistance(\varphi_1,\varphi_1) - \frac{1}{2} \tensordistance(\varphi_2,\varphi_2) \\
			&= \int_{\partial_e \mathcal{T}(C(S_n^+)) \times \partial_e \mathcal{T}(C(S_n^+))} \frac{1}{2n} \sum_{i,j=1}^n (\tau_1(u_{ij}) - \tau_2(u_{ij}))^2\,d\mu(\tau_1,\tau_2) \\
			&= \int_{B_n \times B_n} \frac{1}{2n} \sum_{i,j=1}^n (A_{ij} - A_{ij}')^2 \,d\nu(A,A') \\
			&= \int_{B_n \times B_n} \frac{1}{2} \norm{A - A'}_2^2 \,d\nu(A,A') \\
			&\geq \frac{1}{2} W_2(\nu_1,\nu_2)^2.
		\end{align*}
		
		To prove the opposite inequality, let $\nu \in \mathcal{P}(B_n \times B_n)$ be an optimal coupling of $\nu_1$ and $\nu_2$ for $W_2$.  To prove the lemma, it suffices to show that $\nu = (\omega \times \omega)_* \mu$ for some coupling $\mu$ of $(\mu_1,\mu_2)$ on $\partial_e \mathcal{T}(C(S_n^+)) \times \partial_e \mathcal{T}(C(S_n^+))$, because then the computations that we did above can be applied to this $\mu$.
		
		For $j = 1$, $2$, define a measure $\rho_j \in \mathcal{P}(\partial_e \mathcal{T}(C(S_n^+)) \times B_n)$ by $\rho_j = (\id, \omega)_* \mu_j$.  By construction the marginals of $\rho_j$ are $\mu_j$ and $\nu_j$.  We also have a measure $\nu$ on $B_n \times B_n$ whose marginals are $\nu_1$ and $\nu_2$.  By taking conditionally independent joins or fibred products of the probability spaces, there exists a probability measure $\rho$ on $\partial_e \mathcal{T}(C(S_n^+)) \times B_n \times \partial_e \mathcal{T}(C(S_n^+)) \times B_n$ such that
		\begin{itemize}
			\item The marginal of $\rho$ on the first two coordinates is $\rho_1$.
			\item The marginal of $\rho$ on the second and fourth coordinates (the copies of $B_n$) is $\nu$.
			\item The marginal of $\rho$ on the last two coordinates is $\rho_1$.
		\end{itemize}
		Let $\mu \in \mathcal{P}(\partial_e \mathcal{T}(C(S_n^+)) \times \partial_e \mathcal{T}(C(S_n^+)))$ be the marginal of $\rho$ on the first and third coordinates.  Thus, the marginals of $\mu$ are the marginals of $\rho$ on the first and third coordinates respectively, which are $\mu_1$ and $\mu_2$.  So it remains to show that $(\omega \times \omega)_* \mu = \nu$.  For this, note that $\rho_1$ and $\rho_2$ are supported on $\{(\tau,A) \in \partial_e \mathcal{T}(C(S_n^+)) \times B_n: \omega(\tau) = A\}$.  Hence, $\rho$ is supported on
		\[
		\{ (\tau,A,\tau',A'): \omega(\tau) = A, \omega(\tau') = A' \}.
		\]
		Therefore, for each Borel set $E \subseteq B_n \times B_n$, we have
		\begin{align*}
			\nu(E) &= \rho(\{(\tau,A,\tau',A'): (A,A') \in E\}) \\
			&= \rho(\{ (\tau,A,\tau',A'): \omega(\tau) = A, \omega(\tau') = A', (A,A') \in E \}) \\
			&= \rho(\{ (\tau,A,\tau',A'): \omega(\tau) = A, \omega(\tau') = A', (\tau,\tau') \in (\omega \times \omega)^{-1}(E) \}) \\
			&= \rho(\{ (\tau,A,\tau',A'):  (\tau,\tau') \in (\omega \times \omega)^{-1}(E) \}) \\
			&= \mu((\omega \times \omega)^{-1}(E)).
		\end{align*}
		Hence, $\nu = (\omega \times \omega)_* \mu$ as desired.
	\end{proof}

	\subsection{Behavior under convolution} \label{subsec: tensor convolution}
	
	In this section, we show the subadditivity property for the distance with respect to convolution.  This is the dual version or quantum analog of the multiplication operation being $1$-Lipschitz in each argument for a metric group.
	
	We define the following $*$-homomorphism using the Sweedler notation
	\[
	\tilde{\Delta}: A \otimes A \to A \otimes A \otimes A \otimes A \text{ defined as }
	\]
	\[ 
	\tilde{\Delta}(a \otimes 1)  = \sum_{(a)} a_{(1)} \otimes 1 \otimes a_{(2)} \otimes 1 = \iota_{1,3}(\Delta(a) \otimes 1) \]
	\[
	\tilde{\Delta}(1 \otimes a) = \sum_{(a)} 1 \otimes a_{(1)} \otimes 1 \otimes a_{(2)} = \iota_{2,4}( 1\otimes \Delta(a))
	\] 
	where $\iota_{i,j} : A \otimes A \otimes A \otimes A \to A \otimes A \otimes A \otimes A$ are embeddings to the $(i,j)$-th legs.
	
	\begin{lemma}\label{comult-inequality}
		For the Hamming cost $C_H \in (A \otimes A)_{+}$, we have 
		\[
		\tilde{\Delta}(C_H) \leq \iota_{1,2}(C_H) + \iota_{3,4}(C_H)
		\]
	\end{lemma}
	\begin{proof}
		\begin{align*}
			\tilde{\Delta}(C_H) & =\tilde{\Delta} \left(1 - \frac{1}{n} \sum_{i,j=1}^n u_{ij} \otimes u_{ij} \right) \\
			& = 1 \otimes 1 \otimes 1 \otimes 1 - \frac{1}{n} \sum_{i,j,k,l=1}^n u_{ik} \otimes u_{il} \otimes u_{kj} \otimes u_{lj} \\
			&  \leq 1 \otimes 1 \otimes 1 \otimes 1   - \frac{1}{n} \sum_{i,j,k=1}^n u_{ik} \otimes u_{ik} \otimes u_{kj} \otimes u_{kj} \\
			&  = 1 \otimes 1 \otimes 1 \otimes 1 - \frac{1}{n} \sum_{i,k=1}^n u_{ik} \otimes u_{ik} \otimes 1 \otimes 1 + \frac{1}{n} \sum_{i,k=1}^n u_{ik} \otimes u_{ik} \otimes 1 \otimes 1\\
			&
			- \frac{1}{n} \sum_{i,j,k=1}^n u_{ik} \otimes u_{ik} \otimes u_{kj} \otimes u_{kj} \\
			& = \iota_{1,2} (C_H) + \frac{1}{n} \sum_{i,k=1}^n u_{ik} \otimes u_{ik} \otimes 1 \otimes 1 \left( 1 - \sum_{j=1}^n 1 \otimes 1 \otimes u_{kj} \otimes u_{kj}
			\right) \\
			& \leq \iota_{1,2}(C_H) + \iota_{3,4}(C_H).
		\end{align*}
	\end{proof}

	\begin{proposition} \label{prop: convolution for tensor distance}
		Let $\phi_1, \phi_2, \psi_1, \psi_2 \in \cT(A)$. Then 
		\[
		\tensordistance(\phi_1 * \psi_1, \phi_2 * \psi_2) \leq \tensordistance(\phi_1, \phi_2) + \tensordistance(\psi_1, \psi_2).
		\]
	\end{proposition}
	
	\begin{proof}
		Let $\phi, \psi$ be the tracial tensor couplings for $\phi_1, \phi_2$ and $\psi_1, \psi_2$ respectively. We define $\phi \tilde{*} \psi = (\phi \otimes \psi) \circ \tilde{\Delta} \in \cT (A \otimes A)$ and show that it is a tracial tensor coupling for the convolutions $\phi_1 * \psi_1$ and $\phi_2 * \psi_2$. For any $a \in A$, we have 
		\begin{align*}
			\phi \tilde{*} \psi (a \otimes 1) & = (\phi \otimes \psi) \tilde{\Delta}(a \otimes 1) \\
			& = (\phi \otimes \psi) \left( \sum_{(a)} a_{(1)} \otimes 1 \otimes a_{(2)} \otimes 1 \right)\\
			& = \sum_{(a)} \phi(a_{(1)} \otimes 1) \psi(a_{(2)} \otimes 1) \\
			& = \sum_{(a)} \phi_1(a_{(1)}) \psi_1(a_{(2)})\\
			& = (\phi_1 \otimes \psi_1 )(\Delta(a)) = \phi_1 * \psi_1 (a).
		\end{align*}
		Similarly, we get $\phi \tilde{*} \psi (1 \otimes a) = \phi_2 * \psi_2 (a)$ for all $a \in A$. For any such coupling of the convolutions $\phi_1 * \psi_1 $ and $\phi_2 * \psi_2$, we have the following inequalty.
		Since $\phi \otimes \psi $ is positive, by \cref{comult-inequality}, we get
		\begin{align*}
			\phi \tilde{* } \psi (C_H) & = (\phi \otimes \psi) (\tilde{\Delta}(C_H))\\
			& \leq \phi \otimes \psi (\iota_{1,2}(C_H)) + \phi \otimes \psi (\iota_{3,4} (C_H))\\
			& = \phi (C_H) \psi(1 \otimes 1) + \phi(1 \otimes 1) \psi (C(H)) = \phi(C_H) + \psi(C_H).
		\end{align*}
		
		Now, taking the infimum over all the couplings of the convolutions gives us the required inequality
	\end{proof}

	\section{Connections and future directions}
	
	\subsection{Comparison of distances and recovery of classical case} \label{subsec: comparison}
	
	The next proposition gives inequalities relating the different versions of the Hamming distance that we have defined.
	
	\begin{proposition} \label{prop: comparison of distances}
		For $\varphi_1$, $\varphi_2 \in \mathcal{T}(C(S_n^+))$, we have
		\[
		\ldistance(\varphi_1,\varphi_2) \leq \freedistance(\varphi_1,\varphi_2) \leq \tensordistance(\varphi_1,\varphi_2) \leq 1.
		\]
	\end{proposition}
	
	\begin{proof}
		To show that $\ldistance(\varphi_1,\varphi_2) \leq \freedistance(\varphi_1,\varphi_2)$, consider a tracial coupling $(M,\tau,\alpha_1,\alpha_2)$.  Observe that
		\begin{align*}
			\norm{\alpha_1(u_{ij}) - \alpha_2(u_{ij})}_{L^1(M,\tau)} &\leq \norm{\alpha_1(u_{ij}) - \alpha_1(u_{ij}) \wedge \alpha_2(u_{ij})}_{L^1(M,\tau)} +  \norm{\alpha_2(u_{ij}) - \alpha_1(u_{ij}) \wedge \alpha_2(u_{ij})}_{L^1(M,\tau)} \\
			&= \tau[\alpha_1(u_{ij}) - \alpha_1(u_{ij}) \wedge \alpha_2(u_{ij})] + \tau[\alpha_2(u_{ij}) - \alpha_1(u_{ij}) \wedge \alpha_2(u_{ij})] \\
			&= \tau(\alpha_1(u_{ij})) + \tau(\alpha_2(u_{ij})) - 2 \tau(\alpha_1(u_{ij}) \wedge \alpha_2(u_{ij}))
		\end{align*}
		since $\alpha_1(u_{ij}) \geq \alpha_1(u_{ij}) \wedge \alpha_2(u_{ij})$ and the same for $\alpha_2(u_{ij})$.  Since $\frac{1}{n} \sum_{i,j=1}^n u_{ij} = 1$, this implies that
		\[
		\frac{1}{n} \sum_{i,j=1}^n \norm{\alpha_1(u_{ij}) - \alpha_2(u_{ij})}_{L^1(M,\tau)} \leq 2 \left( 1 - \frac{1}{n} \sum_{i,j=1}^n \tau(\alpha_1(u_{ij}) \wedge \alpha_2(u_{ij})) \right).
		\]
		Since the tracial coupling was arbitary, we obtain the desired inequality upon dividing by $2$.
		
		To show that $\freedistance(\varphi_1,\varphi_2) \leq \tensordistance(\varphi_1,\varphi_2)$, let $\varphi \in \Pi_{\tr,\otimes}(\varphi_1,\varphi_2)$.  Let $(H_\varphi,\pi_\varphi)$ be the GNS construction associated to $A \otimes A$ and $\varphi$.  Let $M = \pi_{\varphi}(A \otimes A)''$ and let $\tau$ be the vector state on $B(H_\varphi)$ given by $\widehat{1}$.  By Lemma \ref{lem: GNS tracial von Neumann algebra}, $(M,\tau)$ is a tracial von Neumann algebra.  Let $\alpha_1$, $\alpha_2: A \to M$ be the inclusions of the first and second tensorands into $A \otimes A$ composed with the map $\pi_{\varphi}: A \otimes A \to M$.  Then $(M,\tau,\alpha_1,\alpha_2)$ is a tracial coupling in the sense of Definition \ref{def: tracial coupling}.  Therefore,
		\[
		\freedistance(\varphi_1,\varphi_2) \leq 1 - \frac{1}{n} \sum_{i,j=1}^n \tau(\alpha_1(u_{ij}) \wedge \alpha_2(u_{ij})).
		\]
		Since $\alpha_1(A)$ and $\alpha_2(A)$ commute, $\alpha_1(u_{ij}) \wedge \alpha_2(u_{ij}) = \alpha_1(u_{ij}) \alpha_2(u_{ij})$.  Thus,
		\[
		\freedistance(\varphi_1,\varphi_2) \leq 1 - \frac{1}{n} \sum_{i,j=1}^n \tau(\alpha_1(u_{ij}) \alpha_2(u_{ij})) = 1 - \frac{1}{n} \sum_{i,j=1}^n \varphi(u_{ij} \otimes u_{ij}).
		\]
		Taking the infimum over $\varphi$, we have $\freedistance(\varphi_1,\varphi_2) \leq \tensordistance(\varphi_1,\varphi_2)$ as desired.
		
		The inequality $\tensordistance(\varphi_1,\varphi_2) \leq 1$ is immediate because the cost operator $C_H$ satisfies $C_H \leq 1$.
	\end{proof}
	
	We can further compare the distances defined in this paper with the norm $\norm{\varphi_1 - \varphi_2}_{C(S_n^+)^*}$, where $C(S_n^+)^*$ denotes the dual space.  This is motivated by the fact that Wasserstein distance is bounded by a constant time the total variation distance for probability measures on a compact metric space.  Recall that probability measures on a compact Hausdorff space $X$ are equivalent to states on $C(X)$.  The total variation distance of two probability measures is $d_{TV}(\mu,\nu) = \sup_{A \subseteq X} \{|\mu(A) - \nu(A)|\}$ where $A$ ranges over Borel subsets of $X$, and this agrees with $(1/2) \norm{\mu - \nu}_{C(X)^*}$.  Hence, for states $\varphi_1$, $\varphi_2$ on a $\mathrm{C}^*$-algebra $A$, the quantity $(1/2) \norm{\varphi_1 - \varphi_2}_{A^*}$ serves as an analog of the total variation distance.

	\begin{proposition} \label{prop: comparison TV}
		For $\varphi_1$, $\varphi_2 \in \mathcal{T}(C(S_n^+))$, we have
		\[
		\freedistance(\varphi_1,\varphi_2) \leq \frac{1}{2} \norm{\varphi_1 - \varphi_2}_{C(S_n^+)^*}.
		\]
	\end{proposition}

	\begin{proof}
		We begin by recalling the analog of Jordan decomposition for tracial functionals.  Suppose $A$ is a $\mathrm{C}^*$-algebra and $\varphi \in A^*$ is self-adjoint ($\varphi(a^*) = \overline{\varphi(a)}$) and tracial.  By \cite[Proposition 2.8]{CuntzPedersen1979}, $\varphi$ has a unique decomposition as $\varphi_+ - \varphi_-$, where $\varphi_+$ and $\varphi_-$ are positive and tracial with $\norm{\varphi} = \norm{\varphi_+} + \norm{\varphi_-}$.  Furthermore, the proof gives the following description of $\varphi_+$.  For $x, y \in A_+$, write $x \sim y$ if there exist elements $(u_n)_{n \in \N}$ in $A$ such that
		\[
		x = \sum_{n \in \N} u_n^*u_n, \qquad y = \sum_{n \in \N} u_nu_n^*,
		\]
		where the sums are norm-convergent (see \cite[p. 136]{CuntzPedersen1979}).  Then for $x \in A_+$,
		\begin{align*}
			\varphi_+(x) &= \sup \{ \varphi(y): y, z \in A_+, y + z \sim x\} \\
			\varphi_-(x) &= \sup \{ -\varphi(y): y, z \in A_+, y + z \sim x\}.
		\end{align*}
		Now suppose that $\varphi = \varphi_1 - \varphi_2$ where $\varphi_1$, $\varphi_2$ are tracial states on $A$ (note that such a $\varphi$ is automatically tracial and self-adjoint).  In this case, if $y + z \sim x$, then
		\[
		\varphi(y) = \varphi_1(y) - \varphi_2(y) \leq \varphi_1(y) \leq \varphi_1(y+z) = \varphi_1(x).
		\]
		Taking the supremum over such $y$, $z$, we obtain $\varphi_+(x) \leq \varphi_1(x)$.  By symmetrical reasoning $\varphi_- \leq \varphi_2$.
		
		Hence, for tracial states $\varphi_1$, $\varphi_2$ on $A$, there exist tracial positive functionals $\varphi_+$ and $\varphi_-$ with
		\begin{align*}
			\varphi_1 - \varphi_2 &= \varphi_+ - \varphi_-, & \varphi_+ &\leq \varphi_1, \\
			\norm{\varphi_1 - \varphi_2}_{A^*} &= \norm{\varphi_+}_{A^*} + \norm{\varphi_-}_{A^*}, & \varphi_- &\leq \varphi_2,
		\end{align*}
		and furthermore,
		\[
		\norm{\varphi_+}_{A^*} - \norm{\varphi_-}_{A^*} = \varphi_+(1) - \varphi_-(1) = \varphi_1(1) - \varphi_2(1) = 0,
		\]
		whence
		\[
		\norm{\varphi_+}_{A^*} = \norm{\varphi_-}_{A^*} = \frac{1}{2} \norm{\varphi_1 - \varphi_2}_{A^*}.
		\]
		
		Now let us specialize to the case of $A = C(S_n^+)$.  Note that if $\varphi_1 = \varphi_2$, then the claimed statement holds with both sides being zero.  Moreover, if $\norm{\varphi_1 - \varphi_2}_{A^*} = 2$, it is also immediate since $\freedistance(\varphi_1,\varphi_2) \leq 1$ by Proposition \ref{prop: comparison of distances}.  Hence, assume without loss of generality that $0 < \norm{\varphi_1 - \varphi_2}_{A^*} < 2$, so that $0 < \norm{\varphi_+}_{A^*} < 1$.  Let
		\[
		\psi = \varphi_1 - \varphi_+ = \varphi_2 - \varphi_-.
		\]
		Let $\overline{\psi} = \psi / \norm{\psi}$ and define $\overline{\varphi}_+$ and $\overline{\varphi}_-$ analogously.  Thus, $\overline{\psi}$, $\overline{\varphi}_+$, and $\overline{\varphi}_-$ are tracial states on $A$.  Let $N = \pi_{\overline{\psi}}(A)''$, $M_+ = \pi_{\overline{\varphi}_+}(A)''$, and $M_- = \pi_{\overline{\varphi}_-}(A)''$ be the corresponding von Neumann algebras.  Then let
		\[
		M = N \oplus (M_+ \overline{\otimes} M_-), \qquad \tau = \psi(1) \overline{\psi} \oplus \varphi_+(1) (\overline{\varphi}_+ \otimes \overline{\varphi}_-).
		\]
		Let $\alpha_1, \alpha_2: A \to M$ be given by
		\begin{align*}
			\alpha_1(a) &= \pi_{\overline{\psi}}(a) \oplus (\pi_{\overline{\varphi}_+}(a) \otimes 1)\\
			\alpha_2(a) &= \pi_{\overline{\psi}}(a) \oplus (1 \otimes \pi_{\overline{\varphi}_-}(a)).
		\end{align*}
		Thus, $(M,\tau,\alpha_1,\alpha_2)$ is a tracial coupling of $\varphi_1 = \psi + \varphi_+$ and $\varphi_2 = \psi + \varphi_-$.  Note that
		\[
		\alpha_1(u_{ij}) \wedge \alpha_2(u_{ij}) \geq \pi_{\overline{\psi}}(u_{ij}) \oplus 0.
		\]
		Hence,
		\begin{multline*}
			1 - \frac{1}{n} \sum_{i,j=1}^n \tau(\alpha_1(u_{ij}) \wedge \alpha_2(u_{ij})) \leq 1 - \frac{1}{n} \sum_{i,j=1}^n \psi(1) \overline{\psi}(u_{ij}) = 1 - \frac{1}{n} \sum_{i=1}^n \psi(1) \\
			= 1 - \psi(1) = \varphi_+(1) = \frac{1}{2} \norm{\varphi_1 - \varphi_2}_{C(S_n^+)^*},
		\end{multline*}
		which proves the asserted inequality.
	\end{proof}
	
	Every probability measure on $S_n$ gives rise to a state on $C(S_n)$ and hence a state on $C(S_n^+)$ via the quotient map $\theta$ from \ref{theorem:quotientmap}. We will now show that the non-commutative Wasserstein distance $\freedistance(\varphi_1,\varphi_2)$ agrees with the classical $L^1$ Wasserstein distance of $\mu_1$ and $\mu_2$ with respect to the normalized Hamming distance. To prove this, we must extract tracial couplings of the states from classical couplings of the measures and vice versa.  To obtain a classical coupling from a tracial coupling a von Neumann algebra $M$ that is not necessarily commutative, we use the trick of separating the left and right multiplication as in \cite[Theorem 1.5]{BianeVoiculescu2001}, closely related to Connes' joint distribution trick.
	
	\begin{proposition} \label{prop: recovery of classical distance}
		Let $\mu_1$ and $\mu_2 \in \mathcal{P}(S_n)$, and let $\varphi_1$, $\varphi_2$ be the corresponding traces on $C(S_n^+)$ given by
		\[
		\varphi_j(a) = \int_{S_n} \theta(a) \,d\mu_j.
		\]
		Let $W_1$ be the classical $L^1$-Wasserstein distance on $\mathcal{P}(S_n)$ associated to the normalized Hamming distance $d_H$ on $S_n$.  Then
		\[
		\ldistance(\varphi_1,\varphi_2) = \freedistance(\varphi_1,\varphi_2) = \tensordistance(\varphi_1,\varphi_2) = W_1(\mu_1,\mu_2).
		\]
	\end{proposition}
	
	\begin{proof}
		By Proposition \ref{prop: comparison of distances}, it suffices to show that
		\[
		\tensordistance(\varphi_1,\varphi_2) \leq W_1(\mu_1,\mu_2)
		\]
		and
		\[
		W_1(\mu_1,\mu_2) \leq \ldistance(\varphi_1,\varphi_2).
		\]
		For the first claim, suppose that $\mu$ is a classical coupling of $\mu_1$ and $\mu_2$.  Let $\varphi \in \mathcal{T}(C(S_n^+) \otimes C(S_n^+))$ be the state given by $\varphi(a) = \int_{S_n \times S_n}(\theta \otimes \theta)(a)\,d\mu$.  Then $\varphi(a \otimes 1) = \varphi_1(a)$ and symmetrically for $\varphi_2$.  We also have that
		\[
		\varphi(C_H) = \int_{S_n \times S_n} \left(1 - \frac{1}{n} \sum_{i,j=1}^n \theta(u_{ij}) \otimes \theta(u_{ij}) \right)\,d\mu = \int_{S_n \times S_n} d_H\,d\mu.
		\]
		Since $\mu$ was arbitrary, $\tensordistance(\varphi_1,\varphi_2) \leq W_1(\mu_1,\mu_2)$ which completes the proof.
		
		For the second claim, consider a tracial coupling $(M,\tau,\alpha_1,\alpha_2)$ of $\varphi_1$ and $\varphi_2$.  In order to extract a classical coupling from this, we will first reinterpret $(M,\tau,\alpha_1,\alpha_2)$ as a tracial coupling of states on the commutative $C(S_n)$.  Let $\overline{\varphi}_i$ be the state on $C(S_n)$ given by the measure $\mu_i$, so that $\varphi_i = \overline{\varphi}_i \circ \theta$.  Note that the image $\alpha_i (C(S_n^+))$ is $*$-isomorphic to image of the GNS representation $\pi_{\varphi_i} (C(S_n^+))$ using Lemma \ref{lem: quotient GNS isomorphism} (3) and the faithfulness of $\tau$, and the state $\tau|_{\alpha_i(C(S_n^+))}$ corresponds to $\varphi_i$. Moreover, by Lemma \ref{lem: quotient GNS isomorphism} (2), $\theta$ passes to an isomorphism from $\pi_{\varphi_i}(C(S_n^+))$ to $\pi_{\overline{\varphi}_i}(C(S_n))$ which also preserves the given states.  Hence, $\alpha_i(C(S_n^+))$ is isomorphic to $\pi_{\overline{\varphi}_i}(C(S_n))$, which implies that the map $\alpha_i$ factors through $C(S_n)$, or $\alpha_i = \overline{\alpha}_i \circ \theta$ for some $*$-homomorphism $\overline{\alpha}_i: C(S_n) \to M$; moreover, $\overline{\varphi}_i = \tau \circ \overline{\alpha}_i$. Therefore, $(M, \tau, \overline{\alpha}_1, \overline{\alpha}_2)$ is a tracial coupling of $\overline{\varphi}_1$ and $\overline{\varphi}_2$.
		
		Next, we construct a classical coupling out of $(M,\tau,\overline{\alpha}_1,\overline{\alpha}_2)$.  Recall from \S \ref{subsec: vN basic} that the left and right multiplication representations $\lambda: M \to B(L^2(M,\tau))$ and $\rho: M^{\op} \to B(L^2(M,\tau))$ commute.  Hence, $\lambda \circ \overline{\alpha}_1$ and $\rho \circ \overline{\alpha}_2$ are commuting representations of $C(S_n)$ and $C(S_n^{\op})$ on $L^2(M,\tau)$, and of course $C(S_n)^{\op} = C(S_n)$ since $C(S_n)$ is commutative.  Thus, $(\lambda \circ \alpha_1) \otimes (\rho \circ \alpha_2)$ gives a $*$-representation of $C(S_n) \otimes C(S_n)$ on $L^2(M,\tau)$.  Let $\overline{\varphi}$ be the state on $C(S_n) \otimes C(S_n)$ given by
		\[
		\overline{\varphi}(x) = \ip{\widehat{1}_M,(\lambda \circ \alpha_1) \otimes (\rho \circ \alpha_2)(x) \widehat{1}_M}_{L^2(M,\tau)},
		\]
		and let $\mu$ be the measure on $S_n \times S_n$ corresponding to the state $\overline{\varphi}$.  Note that given $a, b \in C(S_n)$,   
		\[
		\overline{\varphi}(a \otimes b) = \ip{\widehat{1}, \lambda(\alpha_1(a))\rho(\alpha_2(b)) \widehat{1}} = \ip{\widehat{1}, \widehat{\alpha_1(a) 1 \alpha_2(b) }} = \tau(\overline{\alpha}_1(a)\overline{\alpha}_2(b)).
		\]
		In particular, for $a \in C(S_n)$,
		\[
		\int_{S_n \times S_n} a \otimes 1 \, d\mu = \overline{\varphi}(a \otimes 1) = \tau(\overline{\alpha}_1(a)) = \overline{\varphi}_1 (a) = \int_{S_n} a \, d \mu_1,
		\]
		which shows that the first marginal of $\mu$ is $\mu_1$, and symmetrical the second marginal is $\mu_2$.  Hence, $\mu$ is a classical coupling of $(\mu_1,\mu_2)$.
		
		The cost of the classical coupling $\mu$ is given by 
		\begin{align*}
			\int_{S_n \times S_n} d_H(\sigma, \sigma') \, d\mu = 1 - \frac{1}{n} \sum_{i,j}^n \overline{\varphi}(\theta(u_{ij})\otimes \theta(u_{ij})) &= 1 - \frac{1}{n} \sum_{i,j=1}^n \tau(\overline{\alpha}_1(\theta(u_{ij})) \overline{\alpha}_2(\theta(u_{ij}))) \\
			&= 1 - \frac{1}{n} \sum_{i,j=1}^n \tau(\alpha_1(u_{ij}) \alpha_2(u_{ij})) \\
			&= \frac{1}{n} \sum_{i,j=1}^n \left( \frac{1}{2} \tau(\alpha_1(u_{ij})) + \frac{1}{2} \tau(\alpha_2(u_{ij}) - \tau(\alpha_1(u_{ij}) \alpha_2(u_{ij})) \right) \\
			&= \frac{1}{n} \sum_{i,j=1}^n \frac{1}{2} \norm{\alpha_1(u_{ij}) - \alpha_2(u_{ij})}_{L^2(M,\tau)}^2.
		\end{align*}
		Since $\alpha_1(u_{ij})$ and $\alpha_2(u_{ij})$ are nonnegative operators, the Powers-St{\o}rmer inequality shows that
		\[
		\norm{\alpha_1(u_{ij}) - \alpha_2(u_{ij})}_{L^2(M,\tau)}^2 \leq \norm{\alpha_1(u_{ij})^2 - \alpha_2(u_{ij})^2}_{L^1(M,\tau)} = \norm{\alpha_1(u_{ij}) - \alpha_2(u_{ij})}_{L^1(M,\tau)},
		\]
		and therefore
		\[
		W_1(\mu_1,\mu_2) \leq \int_{S_n \times S_n} d_H(\sigma, \sigma') \, d\mu \leq \frac{1}{n} \sum_{i,j=1}^n \frac{1}{2} \norm{\alpha_1(u_{ij}) - \alpha_2(u_{ij})}_{L^1(M,\tau)}.
		\]
		Since the tracial coupling $(M,\tau,\alpha_1,\alpha_2)$ was arbitrary, we obtain the asserted inequality.
	\end{proof}
	
	The next proposition characterizes when each of the distances studied here achieves the maximum value of $1$.  Condition (4) below means that the two states are supported on disjoint sets of entries of the canonical unitary $U = [u_{ij}]$, namely $\{(i,j): \varphi_1(u_{ij}) > 0\}$ and $\{(i,j): \varphi_2(u_{ij}) > 0\}$ are disjoint.  For instance, if $\varphi_j$ was induced by a measure $\mu_j$ on the classical permutation group as in Proposition \ref{prop: recovery of classical distance}, this would mean that each $\sigma_1 \in \supp(\mu_1)$ and $\sigma_2 \in \supp(\mu_2)$ are supported in disjoint sets of indices, or $\sigma_1^{-1} \sigma_2$ has no fixed points.
	
	\begin{proposition} \label{prop: distance one}
		Let $\varphi_1, \varphi_2 \in \mathcal{T}(C(S_n^+))$.  Then the following are equivalent:
		\begin{enumerate}[(1)]
			\item $\ldistance(\varphi_1,\varphi_2) = 1$.
			\item $\freedistance(\varphi_1,\varphi_2) = 1$.
			\item $\tensordistance(\varphi_1,\varphi_2) = 1$.
			\item $\sum_{i,j=1}^n \varphi_1(u_{ij}) \varphi_2(u_{ij}) = 0$.
		\end{enumerate}
	\end{proposition}
	
	\begin{proof}
		(1) $\implies$ (2) $\implies$ (3) is immediate from Proposition \ref{prop: comparison of distances}.
		
		(3) $\implies$ (4), note that $\varphi_1 \otimes \varphi_2$ is a tracial tensor coupling of $\varphi_1$, $\varphi_2$, and hence
		\[
		\tensordistance(\varphi_1,\varphi_2) \leq 1 - \frac{1}{n} \sum_{i,j=1}^n \varphi_1(u_{ij}) \varphi_2(u_{ij}).
		\]
		So if $\tensordistance(\varphi_1,\varphi_2) = 1$, then $\sum_{i,j=1}^n \varphi_1(u_{ij}) \varphi_2(u_{ij}) = 0$.
		
		For (4) $\implies$ (1), assume (4). Let $(M,\tau,\alpha_1,\alpha_2)$ be a tracial coupling of $(\varphi_1,\varphi_2)$.  For every $i,j$, either $\varphi_1(u_{ij}) = 0$ or $\varphi_2(u_{ij}) = 0$.  Thus, either $\tau(\alpha_1(u_{ij})) = 0$ or $\tau(\alpha_2(u_{ij})) = 0$, and hence either $\alpha_1(u_{ij}) = 0$ or $\alpha_2(u_{ij}) = 0$.  Therefore,
		\[
		\norm{\alpha_1(u_{ij}) - \alpha_2(u_{ij})}_{L^1(M,\tau)} = \norm{\alpha_1(u_{ij})}_{L^1(M,\tau)} + \norm{\alpha_2(u_{ij})}_{L^2(M,\tau)} = \tau(\alpha_1(u_{ij})) + \tau(\alpha_2(u_{ij})).
		\]
		Hence,
		\[
		\frac{1}{n} \sum_{i,j=1}^n \frac{1}{2} \norm{\alpha_1(u_{ij}) - \alpha_2(u_{ij})}_{L^1(M,\tau)} = \frac{1}{2n} \sum_{i,j=1}^n [\tau(\alpha_1(u_{ij})) + \tau(\alpha_2(u_{ij}))] = 1,
		\]
		and since $(M,\tau,\alpha_1,\alpha_2)$ was arbitrary, (1) holds.
	\end{proof}
	
	\subsection{Lipschitz seminorms induced by quantum Wasserstein distances} \label{subsec: seminorms}
	
	Recall that a Wasserstein distance on the trace space always induces a corresponding Lipschitz norm on the $\Cst$-algebra. 
	A minimal requirement for a reasonable definition is that the space of Lipschitz elements should be dense in the $\Cst$-algebra, so we now verify this property for $\freedistance$ and $\ldistance$.
	
	\begin{proposition} \label{prop: Lipschitz dense}
		For all indices $i_1$, $j_1$, \dots, $i_\ell$, $j_\ell \in [n]$, we have
		\[
		\norm{u_{i_1j_1} \dots u_{i_\ell j_\ell}}_{\Lip(\freedistance)} \leq \norm{u_{i_1j_1} \dots u_{i_\ell j_\ell}}_{\Lip(\ldistance)} \leq 2n \ell.
		\]
		In particular, the $*$-algebra generated by the $u_{ij}$'s is contained in the domains of the Lipschitz seminorms associated to $\freedistance$ and $\ldistance$ are dense. 
	\end{proposition}
	
	\begin{proof}
		The first inequality follows because $\ldistance \leq \freedistance$.  For the second inequality, let $\varphi_1$, $\varphi_2 \in \mathcal{T}(C(S_n^+))$, and let $(M,\tau,\alpha_1,\alpha_2)$ be a tracial coupling.  Then using the triangle inequality and non-commutative H{\"o}lder's inequality,
		\begin{align*}
			|\varphi_1(u_{i_1j_1} \dots u_{i_\ell j_\ell}) - \varphi_2(u_{i_1j_1} \dots u_{i_\ell j_\ell})| &\leq |\tau(\alpha_1(u_{i_1j_1} \dots u_{i_\ell j_\ell}) - \alpha_2(u_{i_1j_1} \dots u_{i_\ell j_\ell}))| \\
			&\leq \sum_{t=1}^\ell \lVert \alpha_1(u_{i_1j_1} \dots u_{i_{t-1} j_{t-1}})(\alpha_1(u_{i_tj_t}) - \alpha_2(u_{i_tj_t})) \\
			& \qquad \qquad \alpha_2(u_{i_{t+1}j_{t+1}} \dots u_{i_\ell j_\ell}) \rVert_{L^1(M,\tau)} \\
			&\leq \sum_{t=1}^\ell \norm{\alpha_1(u_{i_tj_t}) - \alpha_2(u_{i_tj_t})}_{L^1(M,\tau)} \\
			&\leq \ell \max_{i,j} \norm{\alpha_1(u_{ij}) - \alpha_2(u_{ij})}_{L^1(M,\tau)} \\
			&\leq 2n \ell \cdot \frac{1}{n} \sum_{i,j=1}^n \frac{1}{2} \norm{\alpha_1(u_{ij}) - \alpha_2(u_{ij})}_{L^1(M,\tau)}.
		\end{align*}
		Taking the infimum over the coupling $(M,\tau,\alpha_1,\alpha_2)$ shows that
		\[
		|\varphi_1(u_{i_1j_1} \dots u_{i_\ell j_\ell}) - \varphi_2(u_{i_1j_1} \dots u_{i_\ell j_\ell})| \leq 2n\ell \ldistance(\varphi_1,\varphi_2),
		\]
		which is the desired bound for the Lipschitz norm.
	\end{proof}

	\subsection{Further questions} \label{sec: questions}
	
	We hope that the different distances defined in this paper and these properties are the beginning of further investigations.  In particular, we propose the following questions.
	
	\begin{question}[Duality for Wasserstein distances and Lipschitz seminorms]
		Do the distances $\freedistance$ and $\ldistance$ defined in this paper satisfy the duality $d(\varphi,\psi) = \sup \{ |\varphi(x) - \psi(x)|: \norm{x}_{\Lip} \leq 1\}$?
	\end{question}
	
	\begin{question}[Topologies on $\mathcal{T}(C(S_n^+))$]
		We now know of at least four topologies on $\mathcal{T}(C(S_n^+))$:
		\begin{enumerate}[(1)]
			\item the weak-$*$ topology,
			\item the topology generated by $\ldistance$,
			\item the topology generated by $\freedistance$,
			\item the topology generated by the total variation distance, i.e.\ the restriction of the norm topology of $C(S_n^+)^*$.
		\end{enumerate}
		Each topology in this list is at least as strong as the ones above it,  thanks to Propositions \ref{prop: L distance properties}, \ref{prop: comparison of distances}, and \ref{prop: comparison TV}. We do not yet know if all these topologies are distinct, but this is quite plausible.  For instance, based on \cite[\S 5.5]{GangboJekelNamShlyakhtenko2022}, we expect that the $\ldistance$ would generate a stronger topology than the weak-$*$ topology.  Moreover, since projections $p_k$ can converge to $p$ without $p_k \wedge p$ converging $p$, we expect that $\freedistance$ is much stronger than $\ldistance$.  The computation of $\freedistance$ may also lead to interesting combinatorial problems about how to arrange non-commuting projections in a way that maximizes the pairwise intersections.
	\end{question}
	
	\begin{question}[Behavior under iterated convolution]
		For any faithful state on $C(S_n^+)$, the repeated convolution $\varphi^{*k}$ converges in the weak-$*$ topology as $k \to \infty$ to the Haar state $h$.  Under what conditions does $\varphi^{*k}$ converge to $h$ with respect to $\ldistance$ or $\freedistance$.  Can we describe the limit of $\varphi^{*k}$ when $\varphi$ is not faithful, and determine whether convergence occurs with respect to the metrics in this paper?  Another related question is how the metrics behave with respect to the convolution semigroups studied in \cite{FranzKulaSkalski2016}.  We remark that convergence in total variation distance for various convolution semigroups of states has been shown in \cite{Freslon2019,FTW2022}, though these states are not tracial.
	\end{question}
	
	\begin{question}[Extreme traces on $C(S_n^+)$]
		The study of $\tensordistance$ raises many questions about extreme points in the trace space $\mathcal{T}(C(S_n^+))$.  In particular, can we classify the extreme points that arise from finite-dimensional representations of $C(S_n^+)$?  How does the trace space relate to corepresentations of $S_n^+$?  Is $\mathcal{T}(C(S_n^+))$ a Poulsen simplex?
	\end{question}
	
	\begin{question}[Projection of extreme traces onto Birkhoff polytope]
		Let $B_n = \{A \in M_n(\mathbb{R}): A_{ij} \in [0,1], \sum_i A_{ij} = 1, \sum_j A_{ij} = 1\}$ be the Birkhoff polytope of doubly stochastic matrices.  Consider the map $\omega: \partial_e \mathcal{T}(C(S_n^+))) \to B_n$ by sending $\varphi$ to $[\varphi(u_{ij})]_{i,j=1}^n$.  Under what conditions does $\omega(\varphi) = \omega(\psi)$?  Are there natural sets on which $\omega$ is injective?
	\end{question}
	
	\bibliographystyle{amsalpha}
	\bibliography{bib-quantum-metric-group}
	
\end{document}